\documentclass[10pt]{amsart}
\usepackage[foot]{amsaddr}
\usepackage{graphicx}
\usepackage[nobysame, alphabetic]{amsrefs}
\usepackage{hyperref}

\makeatletter
    
    \@addtoreset{equation}{section}
  \makeatother

\newtheorem{thm}{Theorem}[section]
\newtheorem{coro}[thm]{Corollary}
\newtheorem{prop}[thm]{Proposition}
\newtheorem{lem}[thm]{Lemma}
\theoremstyle{definition}
\newtheorem{defi}[thm]{Definition}
\newtheorem{rmk}[thm]{Remark}
\newtheorem{ex}[thm]{Example}

\newtheorem{question}[thm]{Question}
\newtheorem{assump}[thm]{Assumption}
\newtheorem*{notation*}{Notation}

\newcommand{\Z}{\mathbb N}
\newcommand{\N}{\mathbb N}
\newcommand{\R}{\mathbb R}
\newcommand{\Q}{\mathbb Q}

\newcommand{\mcg}{\mathrm{MCG}(S)}
\newcommand{\PSL}{\mathrm{PSL}(2,\mathbb{C})}
\newcommand{\HH}{\mathbb{H}^{3}}
\newcommand{\AH}{\mathrm{AH}(S)}
\newcommand{\QF}{\mathrm{QF}(S)}
\newcommand{\LQ}{\mathrm{LQ}(S)}
\newcommand{\T}{\mathcal{T}(S)}
\newcommand{\hatT}{\hat{\mathcal{T}}(S)}
\newcommand{\QD}{\mathrm{QD}}
\newcommand{\PFQ}{\partial_{\mathrm{fin}}(\mathcal{Q})}
\newcommand{\PIQ}{\partial_{\infty}(\mathcal{Q})}
\newcommand{\TT}{\mathcal{T}(S)\times\mathcal{T}(S)}
\newcommand{\VR}{V_{R}}
\newcommand{\Vol}{\mathrm{vol}}
\newcommand{\pvh}{\partial_{\mathrm{vh}}}
\newcommand{\Tvh}{\overline{\T}^{\mathrm{vh}}}

\newcommand{\pmf}{\mathcal{PMF}(S)}
\newcommand{\mf}{\mathcal{MF}(S)}
\newcommand{\ext}{\mathrm{Ext}}
\newcommand{\dir}{\mathbf{dir}}

\newcommand{\dt}{d_{\mathcal{T}}}
\newcommand{\dwp}{d_{\mathrm{wp}}}
\newcommand{\pwp}{\partial_{\mathrm{wp}}}
\newcommand{\Lip}{\mathrm{Lip}^{C}_{b}}
\newcommand{\Lipi}{\mathrm{Lip}^{1}_{b}}
\newcommand{\hb}{\partial_{h}}

\newcommand{\qf}{\mathrm{qf}}
\newcommand{\ConT}{6\pi(g-1)}
\newcommand{\Con}{3\sqrt{\pi(g-1)}}
\newcommand{\WPE}{\text{WP-end}}

\title[Compactification and distance on Teichm\"uller space via volume]{Compactification and distance on Teichm\"uller space via renormalized volume}
\author{Hidetoshi Masai}
\address{Department of Mathematics, Tokyo Institute of Technology, 2-12-1, Ookayama, Meguro-ku, Tokyo. 152-8551. Japan}
\email{masai@math.titech.ac.jp}

\begin{document}
\begin{abstract}
We introduce a variant of horocompactification which takes ``directions'' into account.
As an application, we construct a compactification of the Teichm\"uller spaces via 
the renormalized volume of quasi-Fuchsian manifolds.
Although we observe that the renormalized volume itself does not give a distance, the compactification allows us to define a new distance on the Teichm\"uller space.
We show that the translation length of pseudo-Anosov mapping classes with respect to this new distance is precisely the hyperbolic volume of their mapping tori.
A similar compactification via the Weil-Petersson metric is also discussed.
\end{abstract}
\maketitle
\section{Introduction}
On an orientable closed surface $S$ of genus $\geq 2$, the space of complex structures has a one-to-one correspondence with the space of hyperbolic structures.
Those complex or hyperbolic structures together with markings have the rich deformation space 
which is called the Teichm\"uller space, denoted $\mathcal{T}(S)$.
Complex structures and hyperbolic structures reveal similar but different features of $\mathcal{T}(S)$.
The Teichm\"uller (resp. Thurston) distance on $\mathcal{T}(S)$ is defined 
as a measurement of the deformation of complex (resp. hyperbolic) structures.
Similarly, the Gardiner-Masur \cite{GM} (resp. Thurston \cite{FLP}) boundary is a boundary of $\T$ constructed by regarding $\T$ as the space of complex (resp. hyperbolic) structures.

The theory of horoboundary, which is introduced by Gromov \cite{Gromov}, is a universal method to compactify any given metric space.
The horoboundary with respect to the Teichm\"uller distance and the Thurston distance is the Gardiner-Masur boundary \cite{LS} and the Thurston boundary \cite{Wal} respectively.
Thus one observes that the theory of horoboundary relates naturally distances and boundaries defined in the same context.
The main purpose of this paper is to develop a variant of ``horoboundary" to construct a boundary of $\T$ via the renormalized volume of quasi-Fuchsian manifolds.
To discuss the renormalized volume, let us first recall the Bers compactification \cite{Bers}.
With his celebrated simultaneous uniformization, Bers showed that the space $\TT$ parametrizes the space of quasi-Fuchsian manifolds.
Let $M(X,Y)$ denote the quasi-Fuchsian manifold with parameter $(X,Y)\in\TT$.
Fixing the second coordinate $Y$ and considering the Schwarzian derivatives, 
Bers showed that there is an embedding $\T\rightarrow\QD(Y)$ where $\QD(Y)$ is the space of holomorphic quadratic differentials on $Y$.
The closure of this embedding is compact and is called the Bers compactification.
Based on the ideas from Graham-Witten \cite{GW}, several authors (see e.g. \cite{BBB, BC, GMR, Kojima-McShane, KS, Schlenker-MRL, Schlenker}) define and discuss {\em the renormalized volume} of hyperbolic $3$-manifolds.
In particular the renormalized volume of quasi-Fuchsian manifolds $M(X,Y)$ defines a smooth function
$$\VR:\T\times\T\rightarrow\R.$$

Although the function $\VR$ does not define a distance on $\mathcal{T}(S)$ (the triangle inequality does not hold, see \S \ref{sec.not-distance} for more detail),
we may consider horofunctions defined via $\VR$.
Namely, for each $Z\in\T$, we may define a {\em volume horofunction} $\nu_Z:\T\rightarrow\R$ by
$$\nu_Z(X):=\VR(X,Z)-\VR(b,Z)$$
where $b\in\T$ is the fixed base point.
It turns out $\nu_Z$ is $\Con$-Lipschitz (Proposition \ref{prop.Lip}) with respect to the Weil-Petersson metric, where $g$ is the genus of $S$.
Thus we get a map from $\T$ into $\Lip\T$, the space of $C=\Con$-Lipschitz functions which vanishes at $b$.
This map is in fact injective and continuous (Proposition \ref{prop.horo}).
We may take the closure at this point, however, to discuss more properties of $\VR$ it is natural to take the Bers embedding into account.
By the work of Krasnov-Schlenker \cite{KS}, the differential of $\VR$ at $X\in\T$ is expressed 
in terms of the Bers embeddings.
The same proof works for $\nu_Z$ and its differential is given in terms of the Bers embeddings as well (see Proposition \ref{prop.int1}).
We introduce a space that encodes both $\nu_Z$ and Bers embeddings.
We denote the space by $\LQ$ (``L" stands for Lipschitz and ``Q" stands for quadratic differential), see Definition \ref{defi.LQ}.
Then the main result of the paper is the following.
\begin{thm}\label{thm.main-cpt}
	The space $\LQ$	is compact and metrizable, and there is a homeomorphism $\mathcal{V}:\T\rightarrow\LQ$ onto its image, defined via the volume horofunctions and Bers embeddings.
	In particular, the closure $\Tvh:=\overline{\mathcal{V}(\T)}\subset\LQ$ is a compactification of $\T$.
\end{thm}
Here the ``vh" stands for \underline volume and \underline horofunction.
For more detail about Theorem \ref{thm.main-cpt}, see Proposition \ref{prop.LQ} and \ref{prop.Nu}.
With a similar idea, we may also define a compactification via Weil-Petersson metric, see \S \ref{sec.WP}.

For a metric space $(M,d)$, the horofunction $h_z:M\rightarrow\R$ at $z\in M$ is defined as $h_z(x):=d(x,z)-d(b,z)$ here $b\in M$  is a base point.
Then for $x,y\in M$, 
$$\sup_{z\in M}h_z(x)-h_z(y)$$
recovers the original distance $d$ by triangle inequality and the fact $h_y(x)-h_y(y) = d(x,y)$.
Due to the lack of the triangle inequality for $\VR:\TT\rightarrow\R$, the quantity
$$d_R(X,Y):=\sup_{Z\in \T}\nu_Z(X)-\nu_Z(Y)$$
does not coincide with $\VR(X,Y)$.
Interestingly, $d_R$ does satisfy the triangle inequality and 
hence $d_R$ becomes a distance\footnote{$d_R$ may be asymmetric, however, one may also consider its symmetrization.} (Theorem \ref{thm.dv-dist}).
As $\Tvh$ is compact, the supremum in the definition of $d_R$ is realized in $\Tvh$.
One good feature of $\Tvh$ is that the action of the mapping class group $\mcg$ extends continuously on $\Tvh$ (Proposition \ref{prop.bh-mcg-action}).
Similarly to other known compactifications (see \S \ref{sec.TTBcpt}), the action of any pseudo-Anosov map has unique attracting and repelling fixed points on $\Tvh$ (Lemma \ref{lem.RB-ue}). 
Using the north-south dynamics, we establish one remarkable feature of the distance $d_R$.
\begin{thm}\label{thm.TL=HV}
	Let $\psi\in\mcg$ be a pseudo-Anosov mapping class and $M(\psi)=S\times I/(x,1)\sim(\psi(x),0)$ denote the mapping torus of $\psi$.
	Then the translation distance of $\psi$ with respect to $d_R$ is equal to the hyperbolic volume of $M(\psi)$, that is, for any $X\in\T$, we have
	$$\lim_{n\rightarrow\infty}\frac{1}{n}d_R(X,\psi^n X) = \Vol(M(\psi)).$$
\end{thm}
It is worth mentioning that the theory of horofunctions has found interesting applications in ergodic theory, see e.g. \cite{Karlsson-two, KL, MT, Masai}.
In our proof of Theorem \ref{thm.TL=HV}, ergodic theory plays a key role.
See \S \ref{sec.dist} for the detail.
\subsection{(In)finite horofunctions and visual spheres}
With its relation to the theory of 3-dimensional topology and geometry (e.g. the ending lamination conjecture, resolved by Brock-Canary-Minsky \cite{BCM}), Bers compactification reveals interesting relations between 2 and 3-dimensional hyperbolic geometry.
However, one difficulty of studying Bers compactification is its dependence on the base point $Y$.
Moreover, Kerckhoff-Thurston \cite{KT} proved that the action of the mapping class group does not extend continuously to the boundary of the Bers compactification.
It turns out that the dependence of the Bers compactifications on the base point comes from the non complete nature of the renormalized volume.

To understand the situation better, let us consider first the Weil-Petersson (WP) metric on $\T$.
We refer e.g. \cite{Wolpert-book, Wolpert-survey} for details about WP metric.
Let $\dwp$ denote the distance on $\T$ defined by the WP metric.
As is well-known, $(\T,\dwp)$ is not complete and its completion (denoted $\hatT$) is identified with so-called the augmented Teichmuller space \cite{Masur}.
At the same time, $(\T,\dwp)$ is known to be CAT(0) space \cite{Yamada}, which makes it possible to consider the visual sphere around each point $X\in\T$.
It is shown by Brock \cite{Brock2} that the WP visual spheres depend on the base point $X\in\T$ as well.
In WP metric, the distance between $X\in\T$ and $Y_\infty\in\hatT$ is always finite.
Roughly speaking, due to this finiteness, visual spheres depend on the base point.
As a more basic example, let us consider $\R^2$ for a moment.
Imagine the visual sphere at each $p\in\R^2$ which is the set of directions from $p$.
Then if we pick a point $x\in\R^2$, then the direction from $p$ to $x$ apparently depends on $p$.
On the other hand, consider a ray ${\bf r}(t)=(t\cos\theta,t\sin\theta)\subset\R^2$.
The direction $\dir_p(t)\in T^1_p\R^2$ (the unit tangent sphere at $p$, which we identify with the visual sphere) from $p$ to ${\bf r}(t)$ converges to define a point $\dir_p(\bf r)$ in the visual sphere of $p$.
One readily sees that $\dir_p(\bf r)$ does not depend on the base point $p$.
This illustrates a rough picture of $(\T,\dwp)$, see \S \ref{sec.horo-nonproper} for more detailed examples to explain the situation.

Furthermore, non completeness of $\dwp$ implies non properness of $\dwp$ (i.e. closed metric balls are not necessarily compact).
If a metric space is not proper, the theory of horofunctions does not immediately give a compactification (see \S \ref{sec.horo-nonproper}).
To overcome this difficulty, when we compactify $(\T,\dwp)$, we consider not only horofunctions with respect to $\dwp$, but also the visual spheres.
The data of directions in the visual sphere makes the situation much simpler, and one may readily obtain a compactification.
As horofunctions together with the directions look like magnitudes and arguments of polar coordinates, we suggestively call the above compactification a {\em horocoordinate compactification}.
The data of directions turns out to be very natural for horofunctions.

In Bridson-Haefliger \cite{BH}, it is observed that in CAT(0) spaces, such data of directions actually gives differentials of horofunctions (see \S \ref{sec.Angles} for more detail).
One of the main observation of this paper is that the theory of horocoordinates works well with the renormalized volume $V_R$ together with the Bers embeddings.
Since the Bers embeddings take value in the space of quadratic differentials $\QD(S)$, it defines a Weil-Petersson gradient flow (see e.g. \cite{BBB,BBP}).
Hence we may regard Bers embeddings as the space of ``directions".
Moreover as we have mentioned above, the differential of the renormalized volume $\VR$ can be expressed in terms of Bers embeddings.
Hence we may naturally consider horocoordinates.
The horocoordinate compactification with respect to $\VR$ and Bers embeddings is what we obtain in Theorem \ref{thm.main-cpt}.

\subsection{Organization of the paper}
The paper is organized as follows.
Before discussing the Teichm\"uller space $\T$, we first consider horocoordinates.
In \S \ref{sec.horo-coord}, we define the horocoordinate boundary and demonstrate examples.
Then in \S \ref{sec.pre}, we recall basic facts on $\T$ that we need in this paper.
\S\ref{sec.TTBcpt} is devoted to a quick review of known compactifications of $\T$.
In \S \ref{sec.WP}, we construct the horocoordinate compactification of $\T$ with respect to the Weil-Petersson metric.
Several basic properties of its boundary $\pwp\T$ so-obtained are also discussed in \S \ref{sec.WP}.
Our main compactification $\Tvh$ (Theorem \ref{thm.main-cpt}) is defined in \S \ref{sec.main}.
After discussing several properties of $\Tvh$ in \S \ref{sec.main}, we discuss the new distance $d_R$ and prove Theorem \ref{thm.TL=HV} in \S \ref{sec.dist}.

Notations which are not standard in the literature are emphasized as {\bf Notation.}
Although we believe \S \ref{sec.horo-coord} is useful to understand the idea of our compactification, \S \ref{sec.horo-coord} is logically independent of the main section \S \ref{sec.main}.
Hence {\em the readers who are familiar to the Teichm\"uller theory can just check notations in \S \ref{sec.pre} and \S\ref{sec.TTBcpt} and 
go directly to \S \ref{sec.main}} where the compactification of $\T$ via the renormalized volume is defined.

\subsection*{Acknowledgement}
The author would like to thank 
Sadayoshi Kojima,
Hideki Miyachi,
Ken'ichi Ohshika,
Hiroshige Shiga,
Dong Tan,
and Sumio Yamada
for helpful conversations and comments.
The work of the author is partially supported by 
JSPS KAKENHI Grant Number 19K14525.

\section{Horocompactifications}\label{sec.horo-coord}
In this section we consider several spaces which are not proper, and hence the spaces of horofunctions are not compactifications in a strict sense.
Namely, the space of horofunctions has weaker topology than the original space.
The purpose of this section is to demonstrate an idea of how we get a compactification by taking the data of ``directions'' into account.
We first review the notion of horofunctions.
\subsection{Horofunctions}\label{sec.horo}
The idea of horofunctions are first introduced by Gromov \cite{Gromov}.
We only recall basic definitions and properties and omit proofs in this subsection even though most properties are very easy to establish.
See \cite{Gromov, Wal, LS, MT} for more details about horofunctions.

Let $(X,d)$ be a separable metric space. Let $b\in X$ denote a base point which we fix throughout the subsection.
A function $f:X\rightarrow \mathbb{R}$ is called {\em $C$-Lipschitz} if $|f(x)-f(y)|\leq C\cdot d(x,y)$ for all $x,y\in X$.
We define $$\Lip(X):=\{f:X\rightarrow\mathbb{R}\mid f \text{ is }C\text{-Lipschitz and }f(b)=0\}.$$
We equip $\Lip(X)$ with the topology of uniform convergence on compact sets.
Note that by the $C$-Lipschitz property, this topology is equivalent to the topology of pointwise convergence.
Hence we may regard $\Lip(X)$ as a closed subset of $\Pi_{x\in X}[-d(b,x),d(b,x)]$, a compact space by Tychonoff's theorem.
\begin{prop}[see e.g. {\cite[Proposition 3.1]{MT}}]\label{prop.Lip-compact}
Let $(X,d)$ be a separable metric space. Then the space $\Lip(X)$ of $C$-Lipschitz maps is a compact Hausdorff second countable (hence metrizable) space.
\end{prop}
\begin{rmk}
Although in the references, it is only considered the case where $C=1$, the same argument works for any $C>0$.
\end{rmk}
We now define horofunctions.
\begin{defi}
A {\em horofunction} at $z\in X$ is a function $\psi_{z}:X\rightarrow\mathbb{R}$ with
$\psi_{z}(x) = d(x,z)-d(b,z)$ for any $x\in X$.
\end{defi}
By the triangle inequality, we see that $\psi_{z}\in\Lipi(X)$ for every $z\in X$.
Thus, the space $X$ embeds into $\Lipi(X)$.
\begin{lem}[{\cite[Proposition 2.1 and 2.2]{Wal}}]\label{lem.horo}
Let $(X,d)$ be a geodesic separable metric space.
Then the map $\psi:X\rightarrow \Lipi(X)$ defined by $\psi(z):=\psi_{z}$ is continuous and injective.
Furthermore, if $(X,d)$ is a proper space (i.e. every closed metric ball is compact),
then $\psi$ is a homeomorphism onto its image.
\end{lem}
By Proposition \ref{prop.Lip-compact} and Lemma \ref{lem.horo}, the closure $\overline{\psi(X)}$ of $\psi(X)$ in $\Lipi(X)$ is compact.
The space $\partial_{h}X:=\overline{\psi(X)}\setminus\psi(X)$ is called the {\em horoboundary} of $X$.
By abuse of notations, we write $\overline{\psi(X)}$ by $X\cup\hb X$.
For $z\in X\cup\hb X$, we write the associated horofunction by $\psi_{z}$.

Let us consider a group $G$ acting on $X$ by isometries.

\begin{lem}[{\cite[Lemma 3.4]{MT}}]\label{lem.action}
Let $G$ be a group of isometries of $X$. 
Then the action of $G$ on $X$ extends to a continuous action by homeomorphisms on $\overline{\psi(X)}$, defined as
$$g\cdot \psi_{\xi}(z):=\psi_{\xi}(g^{-1}z)-\psi_{\xi}(g^{-1}b)$$
for each $g\in G$ and $\psi_{\xi}\in\overline{\psi(X)}$.
\end{lem}

\subsection{Non proper spaces and horocoordinates}\label{sec.horo-nonproper}
The first example we consider here is the set of rays each of which corresponds to a natural number.
\begin{ex}[see Figure \ref{fig.R}]\label{ex.noncompact1}
For every $n\in \N$, let $R^{+}_{n}$ denote a copy of $\R_{\geq 0}$.
We denote by $r_{n}\in R^{+}_{n}$ the element corresponding to $r\in\R_{\geq 0}$.
Let us consider the space $$\mathcal{R} := \{b\}\sqcup\bigsqcup_{n\in\N} R^{+}_{n}/\sim$$ where $\sim$ is defined so that $b\sim 0_{n}$ for every $n\in\N$.
Equipped with the natural metric inherited from $\R_{\geq 0}$'s,  $\mathcal{R}$ is a metric space which is not proper.
Then we have 
\begin{prop}\label{prop.noncompact1}
The map $$\psi:\mathcal{R}\rightarrow \Lipi(\mathcal{R})$$ is not a homeomorphism onto its image.
\end{prop}
\begin{proof}
We may see this by looking at the image of $\ell_{n}\in R_n^+$ for a fixed $\ell\in\R_{\geq 0}$.
Since $\Lipi(\mathcal{R})$ is compact, the sequence $\{\psi_{\ell_{n}}\}$ has a convergent subsequences with limit $\psi_{\infty}\in \Lipi(\mathcal{R})$.
Let $\psi_{n}:=\psi_{\ell_{n}}$.
For any $r_{k}\in R_k^+\subset\mathcal{R}$ where $k\in\N$, we see that for all $N\geq k$ , we have $\psi_{N}(r_{k}) = d(r_{k}, b)$.
Hence we have $\psi_{\infty}(r_{k}) = d(r_{k}, b)$ which is equal to $\psi_{b}(r_{k})$.
This holds for arbitrarily $r_{k}$, which implies $\psi_{\infty} = \psi_{b}$.
On the other hand the sequence $\{\ell_{n}\}_{n\in\N}$ does not have any convergent subsequence in $\mathcal{R}$.
Hence the inverse of $\psi$ is not continuous.
\end{proof}
\end{ex}

\begin{figure}[h]
	\begin{tabular}{cc}
		 \begin{minipage}[t]{0.45\hsize}
			\centering
	    		\includegraphics[scale = 0.3]{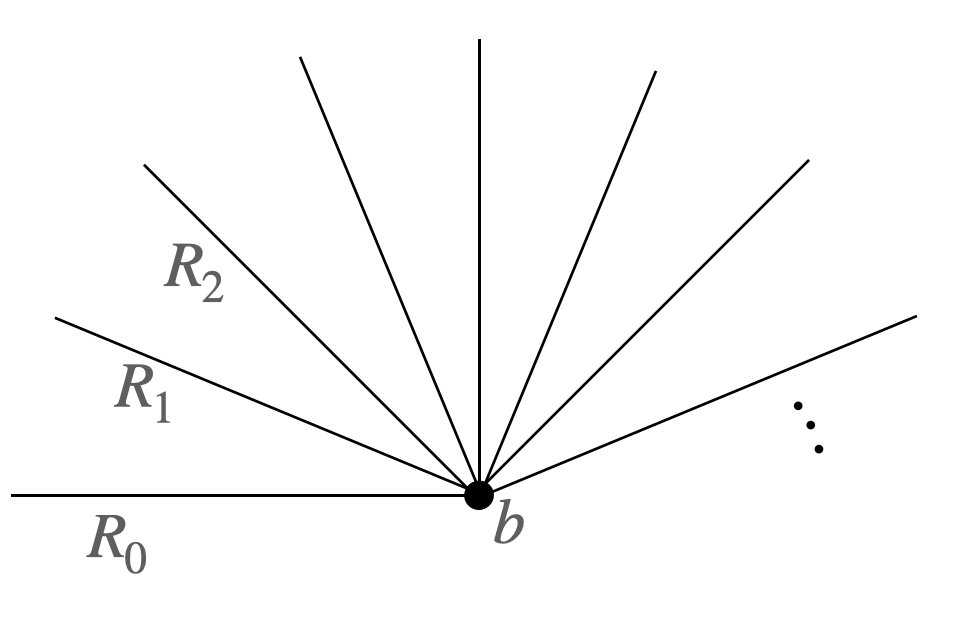}
		  	\caption{$\mathcal{R}$}
	       \label{fig.R}			
		 \end{minipage} 
      \begin{minipage}[t]{0.45\hsize}
	        \centering
	       \includegraphics[scale = 0.3]{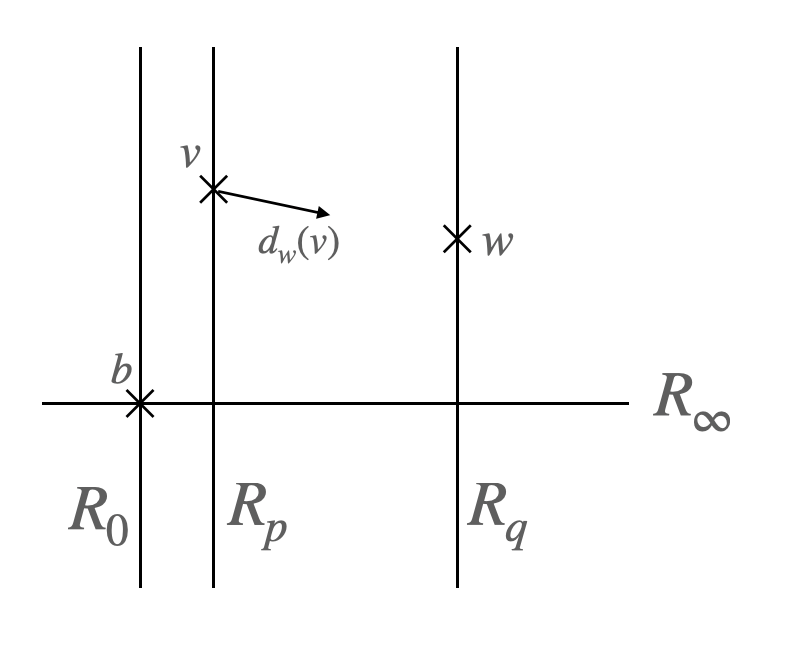}
	       \caption{$\mathcal Q$}
	       \label{fig.Q}
	        \end{minipage}
    \end{tabular}
\end{figure}
The second example illustrates better the situation of the Teichm\"uller space together with distance(-like) functions given by the renormalized volume $\VR$ or the Weil-Petersson metric.
We will observe later that those spaces are not proper and hence the space of ``horofunctions'' may have weaker topology 
as we observed in Example \ref{ex.noncompact1}.
In this paper, by the word {\em compactification}, we require that its restriction to the original space is a homeomorphism.
One may obtain a compactification by introducing certain data of ``directions''.
\begin{ex}
Let $R_{q}$ denote a copy of $\R$ for each $q\in\Q\sqcup\{\infty\}$ and denote $i_q:R_q\rightarrow\R$ the identification.
Let $$\mathcal{Q} := R_{\infty} \sqcup \bigsqcup_{q\in\Q} R_{q}/\sim$$
where we define $\sim$ so that $i_\infty(q)\sim i_q(0)$ for every $q\in \Q$ (see Figure \ref{fig.Q}).
Let $\pi:\mathcal{Q}\rightarrow R_\infty$ denote the map defined by $\pi(R_q) = i_\infty(q)$.
We consider the metric on $\mathcal{Q}$ given by the path metric on $\R$'s.
By considering a sequence of rational numbers converging to an irrational number, the same argument as in Proposition \ref{prop.noncompact1} shows that $\Lipi(\mathcal{Q})$ has weaker topology than $\mathcal{Q}$.
We now suppose that we have an embedding $\iota:\mathcal{Q}\rightarrow\R^{2}$ so that $\iota(R_{\infty})$ is the horizontal axis
and every $\iota(R_{q})$ is a vertical line intersecting with $\iota(R_{\infty})$ at $\iota(i_\infty(q))$.
Then for any given two points $x,y\in\mathcal{Q}$, 
we get a natural direction $\overline \dir_{y}(x)\in T^{1}_{\iota(x)}\R^{2}$ where $T^{1}_{\iota(x)}\R^{2}$ is the unit tangent space at $\iota(x)$.
Now consider a strictly increasing, continuous function $f:\R_{\geq 0}\rightarrow[0,1)$ with $f(0)=0$
 (any such function would do, for example $f(x)=\frac{e^x-1}{e^x+1}$).
 Let $T^{\leq 1}_x\R^2$ denote the space of tangent vectors of length (measured by the standard metric on $\R^2$) $\leq 1$.
Then define $\dir_y(x):=f(|x-y|)\cdot \overline \dir_y(x)\in T^{\leq 1}_x\R^2$.
Note that if $x = y$, we have the zero vector.
For notational simplicity we denote $T^{\leq 1}_{\iota(x)}\R^2$ by $T^{\leq 1}_{x}\R^2$.
We emphasize that the metric on $\mathcal{Q}$ has nothing to do with the embedding $\iota:\mathcal{Q}\rightarrow\R^{2}$.
We define the space $LD(\mathcal{Q})$ ($LD$ stands for \underline Lipschitz and \underline Directions) by
$$LD(\mathcal{Q}):=\prod_{x\in\mathcal{Q}} \left([-d(b,x), d(b,x)]\times T^{\leq 1}_{x}\mathbb{R}^{2}\right).$$
In other words, $LD(\mathcal{Q})$ is the space of sections of a bundle over $\mathcal{Q}$.
$LD(\mathcal{Q})$ is compact by Tychonoff's theorem.
We naturally obtain a map $\Psi:\mathcal{Q}\rightarrow LD(\mathcal{Q})$ by 
$$\Psi(z) := (\psi_{z}(x), \dir_{z}(x))_{x\in\mathcal{Q}}.$$
One may notice some similarity between each $\Psi(z)$ and the polar coordinate.
The first coordinate of $\Psi(z)$, horofunctions, may be seen as the magnitude.
We need some care for the second coordinate of $\Psi(z)$, data of directions. 
For the polar coordinate, directions are defined at the chosen point.
On the other hand, in $\Psi(z)$, $\dir_z(x)$ is the direction from $x$ towards the chosen point $z$.
However, essentially these two are equivalent information.
Hence suggestively, we call each $\Psi(z)$ a {\em horocoordinate} in this paper.
Since $\R^{2}$ is flat, we may simultaneously identify each $T_{x}^{1}(\R^{2})$ with $\mathbb{S}^{1}:=\{\theta\mid \theta\in[0,2\pi)\}$ so that $0\in[0,2\pi)$ corresponds to the positive direction of horizontal lines.
\begin{prop}\label{prop.homeo-image}
The map $\Psi:\mathcal{Q} \rightarrow LD(\mathcal{Q})$ is a homeomorphism onto its image.
\end{prop}
\begin{proof}
By Lemma \ref{lem.horo}, $\Psi$ is injective and continuous to the first coordinate.

Since $\iota:\mathcal{Q}\rightarrow \R^{2}$ is continuous,  
if a sequence $\{z_{n}\}\subset\mathcal{Q}$ converges to $z_{\infty}$, 
then at each $x\in\mathcal{Q}\setminus\{z_\infty\}$, 
corresponding sequence of directions $\overline \dir_{z_{n}}(x)$ converges to $\overline \dir_{z_{\infty}}(x)$, 
and $|\iota(x)-\iota(z_n)|\rightarrow|\iota(x)-\iota(z_\infty)|$ in $\R^2$.
Hence we have $\dir_{z_n}(x)\rightarrow \dir_{z_\infty}(x)$.
If $z_\infty = x$, then $|\iota(x)-\iota(z_n)|\rightarrow 0$, and hence $d_{z_n}(x)$ converges to the zero vector.
Thus continuity of $\Psi$ follows.

Now suppose that a sequence $\{\Psi(z_{n})\}$ converges to $\Psi(z_{\infty})$ for some $z_{\infty}\in \mathcal{Q}$.
Let us pick a point $x\in\mathcal{Q}\setminus\bigcup_{n\in\N\cup\{\infty\}}\{z_n\}$. 
For each $n\in \N\cup\{\infty\}$, the direction $\dir_{z_{n}}(x)$ determines a line $\ell_{n}^{x}$.
Now pick another point $y\in\mathcal{Q}\setminus\bigcup_{n\in\N\cup\{\infty\}}\ell_{n}^{x}$.
Such $y$ exists as there are only countably many $\ell_{n}^{x}$'s.
Then for $n\in \N\cup\{\infty\}$, the line $\ell_{n}^{y}$ determined by $\dir_{z_{n}}(y)$ satisfies $\ell_{n}^{x}\cap\ell_{n}^{y} = z_{n}$.
Since $\dir_{z_{n}}(x)\rightarrow \dir_{z_{\infty}}(x)$ and $\dir_{z_{n}}(y)\rightarrow \dir_{z_{\infty}}(y)$, we have 
$z_{n}\rightarrow z_{\infty}$.
\end{proof}

By Proposition \ref{prop.homeo-image} and the compactness of $LD(\mathcal{Q})$, 
we see that the closure $\overline{\Psi(\mathcal{Q})}$ is a compactification of $\mathcal{Q}$.
Let us denote each element of $\overline{\Psi(\mathcal{Q})}$ by $(\psi_{\xi}(\cdot), \dir_{\xi}(\cdot))$, 
where $\psi_{\xi}\in\Lipi(\mathcal{Q})$ and $\dir_{\xi}(x)\in T^{\leq 1}_{x}(\R^2)$.
Let $\partial\mathcal{Q}:=\overline{\Psi(\mathcal{Q})}\setminus \Psi(\mathcal{Q})$ denote the boundary.
\end{ex}

\subsection{Finite horofunctions}\label{sec.finite-horo}
We now consider a decomposition of the boundary $\partial \mathcal{Q}$. 
Let $$\mathrm{Ray}\mathcal Q:={\{\gamma:[0,\infty)\rightarrow\iota(\mathcal{Q}),~ \gamma(0) = \iota(b), \gamma\text{ is an isometry}\}},$$ and
$\partial_\mathrm{ray}\mathcal{Q}:=\overline{\mathrm{Ray}\mathcal Q}$,
where we take the closure in the space of closed subsets of $\R^2$ with Hausdorff topology.
The space $\partial_\mathrm{ray}\mathcal{Q}$ is naturally identified with 
$\R^+\sqcup\R^-\sqcup\{\pm\infty\}$
 where $\R^+$ (resp. $\R^-$) is a copy of $\R$,
 corresponding to infinite rays starting from $\iota(b)$, travel along $R_\infty$ to $i_\infty(r)$, and then going up (resp. down) vertically.
 The last $\{\pm\infty\}$ correspond to two rays (toward $-\infty$ and $\infty$) on $R_\infty$.
 Each ray naturally determines a point in $\partial\mathcal{Q}$.
 
After taking the data of directions into account, our boundary $\partial\mathcal{Q}$ contains further points than $\partial_{\mathrm{ray}}\mathcal{Q}$.
Let
\begin{align*}
\PFQ&:=\{(\psi, d)\in\partial\mathcal{Q}\mid \inf_{x\in\mathcal{Q}} \psi(x) > -\infty\} \text{ and }\\
\PIQ&:=\partial\mathcal{Q}\setminus\PFQ.
\end{align*}
The first coordinates of elements in $\PFQ$ are called {\em finite horofunctions}.
\begin{prop}\label{prop.QRS}
The closure $\overline{\Psi(\mathcal{Q})}$
is identified with 
$$\mathbb{R}^{2}\sqcup ((\mathbb{S}^{1}\setminus\{\pi/2,3\pi/2\})\cup\partial_\mathrm{ray}\mathcal{Q}).$$
Moreover we have
\begin{align*}
\PFQ &\cong \R^{2}\setminus\iota(\mathcal{Q}),\text{ and } 
\PIQ \cong (\mathbb{S}^{1}\setminus\{\pi/2,3\pi/2\})\cup\partial_\mathrm{ray}\mathcal{Q}.
\end{align*}

\end{prop}
\begin{proof}
We will define a bijection
$B:\overline{\Psi(\mathcal{Q})}\rightarrow \R^{2}\sqcup((\mathbb{S}^{1}\setminus\{\pi/2,3\pi/2\})\cup\partial_\mathrm{ray}\mathcal{Q})$.
Let $\xi = (\psi_{\xi}(\cdot), \dir_{\xi}(\cdot))\in\overline{\Psi(\mathcal{Q})}$.
By definition there is a sequence $\{q_{n}\}\subset\mathcal{Q}$ so that $\Psi(q_{n})$ converges to $\xi$.
Let us denote the angle $\theta_{\xi}(x)\in \mathbb{S}^{1}$ corresponding to the direction $\dir_{\xi}(x)$.

If there exist $x,y\in\mathcal{Q}$ such that $\theta_{\xi}(x) \neq \theta_{\xi}(y)$, 
then the lines determined by the directions $\dir_{\xi}(x)$ and $\dir_{\xi}(y)$ must intersect at some point $r\in\R^{2}$.
Note that for each $q_{n}\in\mathcal{Q}$, $\psi(q_{n})\in LD(\mathcal{Q})$ is characterized so that 
the line determined by the direction $\dir_{q_{n}}(x)$ for every $x\in\mathcal{Q}$ passes through $\iota(q_{n})$.
Hence in this case we have that $\iota(q_{n})$ converges to $r$ and in particular $r$ does not depend on the choice of the sequence $\{q_{n}\}$.
In this case, we let $B(\xi) := r\in\R^{2}$.
Note that for $q\in\mathcal{Q}$, we have $B(q) = \iota(q)$.

If all the angles $\theta_{\xi}(\cdot)\in \mathbb{S}^{1}$ are the same, 
let $\theta:=\theta_{\xi}(\cdot)$.
If $\theta\neq \pi/2, 3\pi/2$ then we define $B(\xi) = \theta\in\mathbb{S}^{1}\setminus\{\pi/2,3\pi/2\}$.
If $\theta = \pi/2, 3\pi/2$, we associate the point in $\partial_\mathrm{ray}\mathcal{Q}$ given as the limit (in the sense of Hausdorff topology) of finite rays connecting $\iota(b)$ and $\iota(q_n)$.
Such limit exists as $\psi_{q_n}\rightarrow\psi_\xi$.
Thus we have defined the map $B:\overline{\Psi(\mathcal{Q})}\rightarrow \R^{2}\sqcup((\mathbb{S}^{1}\setminus\{\pi/2,3\pi/2\})\cup\partial_\mathrm{ray}\mathcal{Q})$.

First note that $B(\mathcal{Q}) = \iota(\mathcal Q)$ and $B$ is injective on $\mathcal Q$.
For each $r=(r_{x}, r_{y})\in\R^{2}\setminus \iota(\mathcal{Q})$, 
there is a sequence $\{q_{n}\}\subset\mathcal{Q}$ so that $\iota(q_{n})\rightarrow r$ in $\R^{2}$.
Then $\Psi(q_{n})$ converges to $(\psi_{\infty}, \dir_{\infty})\in LD(\mathcal{Q})$, 
where $\psi_{\infty} = \psi_{r_{x}}$ and $\dir_{\infty}(x)$ is the direction from $\iota(x)$ to $r$.
Hence $B((\psi_{\infty}, \dir_{\infty})) = r$ and $(\psi_{\infty}, \dir_{\infty})\in\PFQ$. 
Moreover as such $(\psi_{\infty}, \dir_{\infty})$ is characterized so that every line given by $\dir_{\infty}(\cdot)$ passes through $r$, 
the map $B$ is injective on $\PFQ$.
Conversely, if $(\psi_{\infty}, \dir_{\infty})\in\PFQ$ and $\Psi(q_n)\rightarrow(\psi_{\infty}, \dir_{\infty})$, then $d(b,q_n)$ must be finite and hence $\{\iota(q_n)\}$ must converge in $\R^2$.
Therefore we have $\PFQ = B^{-1}(\R^2\setminus\iota(\mathcal Q))$.
This shows the second statement provided $B$ is bijective.

For $\theta\in\mathbb{S}^{1}\setminus\{\pi/2,3\pi/2\}$, there is a sequence $\{q_{n}\}\subset\mathcal{Q}$ so that $\{q_n\}$ leaves every compact set and $\dir_{q_{n}}(b)\rightarrow\theta$.
Then we have $\dir_{q_{n}}(x)\rightarrow\theta$ for every $x\in\mathcal{Q}$ and hence $\theta$ is in the image of $B$.
Furthermore, for each $r_+\in\R^+$, there is a corresponding point $r_\infty\in R_\infty$.
Then there is a sequence $\{s_n\}\subset\mathcal Q$ such that $\pi(s_n)\rightarrow r_\infty$ and
$\dir_{s_n}(b)\rightarrow\pi/2$.
Then the sequence of rays connecting $b$ and $s_n$ converges to the ray corresponding to $r_+$.
The same holds for every $r_-\in\R_-$.
Hence $B$ is surjective.

It remains to prove the injectivity on $\PIQ$.
Suppose $\xi\in\PIQ$ and let $\{t_n\}$ be a sequence so that $\Psi(t_n)\rightarrow\xi$.
Then the directions given by $\xi$ are the same everywhere. 
Let $\theta\in\mathbb{S}^1$ denote the corresponding angle.
Suppose first that $\cos(\theta)>0$. 
In this case, we have $\pi(t_n)\rightarrow\infty$.
Let $q\in\mathcal Q$ be an arbitrary point.
Then for large enough $n>>1$ we have
\begin{itemize}
	\item if $\pi(q)>0$ then the geodesic connecting $b$ and $t_n$ passes through $\pi(q)$, and
	\item if $\pi(q)\leq 0$, then the geodesic connecting $q$ and $t_n$ passes through $b$.
\end{itemize}
Hence the associated horofunction which we denote by $\psi_{\theta}$ is expressed as follows.
\begin{equation*}
\psi_{\theta}(q) =
\begin{cases}
	d(\pi(q), q) - d(\pi(q), b) \text{, if $\pi(q)>0$}\\
	d(q,b)\text{, otherwise.}
\end{cases}
\end{equation*}
In particular $\psi_\theta$ is independent of $\theta$ as long as $\cos\theta>0$.
Hence $\xi$ is determined by the direction $\dir_\xi(b)$, which is equal to $\theta$.
Therefore $B$ is injective at $\xi$. 
Similar arguments shows the injectivity of $B$ at $\xi$ when $\cos\theta<0$ as well.
When $\cos\theta = 0$, in other words $\theta = \pi/2$ or $3\pi/2$.
Then $\pi(t_n)$ must converge to some $r\in R_\infty$, and $\xi$ is determined by $r$ and $\theta$.
Hence we see that $B$ is injective at $\xi$ as well.
\end{proof}

\begin{rmk}
	As demonstrated in \cite{KT, Brock2}, 
the Bers boundaries and the Weil-Petersson visual spheres depend on the base points.
Roughly speaking, these phenomena are due to the dependence on the base points of ``finite'' points.
In fact, certain ``infinite points'' of the Bers compactification and the Weil-Petersson visual spheres are independent of the base points 
(for more detail, see \S \ref{sec.WP} below).
We may observe similar phenomenon for $\partial\mathcal{Q}$.
The ``visual boundary" in this case is the unit tangent sphere $T_{x}^{1}(\R^{2})$ 
in the ``$x$-th'' coordinate of $\overline{\Psi(Q)}$.
As we have seen above if $\xi\in\PFQ$, then the direction $\dir_\xi(x)$ does depend on $x$.
On the other hand, if $\xi\in\PIQ$, all the direction $\dir_\xi(x)$ are the same, and independent of $x$.
Thus, our example $\mathcal{Q}$ illustrates some situation of boundaries of $\T$.
\end{rmk}

\subsection{Angles and distances on Riemannian manifolds of nonpositive curvature}\label{sec.Angles}
A metric space is CAT(0) if every geodesic triangle is thinner than the comparison triangle in the Euclidean space, see \cite{BH} for more details.
Let $(M, g)$ be a (possibly incomplete and non-proper) smooth Riemannian manifold which is CAT(0), or equivalently let us suppose $(M,g)$ is uniquely geodesic and has nonpositive curvature.

We observe that ``directions" have natural relation with horofunctions on $M$.
First let $d:M\times M\rightarrow \R_{\geq 0}$ denote the distance function and 
$\theta_x(v,w)$ the angle between $v,w\in T_x^1(M)$ at $x\in M$, both given by the Riemannian metric $g$.
We fix a base point $b\in M$.
Note that any CAT(0) space is a uniquely geodesic space, i.e. given any two points there exists a unique geodesic connecting the two (see e.g. \cite[Proposition II.1.4]{BH}).
Hence given any two points $x,y\in M$, the direction at $x$ toward $y$ is well-defined in 
$T_x^1(M)$, which we denote by $\overline \dir_y(x)$.
Let $\sigma:[0,T]\rightarrow M$ be a length minimizing geodesic, i.e. $d(\sigma(s),\sigma(t)) = |s-t|$ for any $s,t\in[0,T]$.
In Bridson-Haefliger \cite{BH}, the following variation formula of the distance function is observed.
\begin{prop} [{\cite[Corollary II.3.6]{BH}}]\label{prop.BH-variation}
	Let $\sigma:[0,T]\rightarrow M$ be a length minimizing geodesic and $p:=\sigma(0)$. Then 
	$$\lim_{s\rightarrow 0} \frac{d(\sigma(0),b)- d(\sigma(s),b)}{s} = \cos\theta_p(\dot\sigma(0), \overline \dir_b(p)),$$
	where $\dot\sigma(0)\in T_p^1(M)$ is the unit tangent vector given by $\sigma$.
\end{prop}
As before, we define $\dir_y(x) = f(d(x,y))\overline \dir_y(x)$ where $f(x)=(e^x-1)/(e^x+1)$.
Now let us consider the embedding 
$$\Psi:M\rightarrow \prod_{x\in M} ([-d(b,x), d(b,x)]\times T_x^{\leq 1}(M)),$$
defined by $\Psi(z):= \prod_x (\psi_z(x), \dir_z(x))$, where $\psi_z(\cdot)$ is the horofunction associated to $z$ with base point $b$.
We make the following assumption.
\begin{assump}\label{assump}
	$(M,g)$ is a CAT(0), smooth Riemannian manifold of dimension $\geq 2$,
	and the exponential map from any point is a diffeomorphism from 
	its open domain onto $M$.
\end{assump}
Then we have 
\begin{prop}\label{prop.horo-homeo-riemann}
	If $(M,g)$ satisfies Assumption \ref{assump}, then 
	$\Psi$ is a homeomorphism onto its image, and its closure $\overline{\Psi(M)}$ is compact.
\end{prop}
\begin{proof}
Since $\dir_x(z)\in T_x(M)$ for any $z\in M$, that $\Psi$ on the second coordinate (i.e. $\dir_z(x)$'s ) is continuous follows from the assumption that exponential maps are diffeomorphic.
Also as $\dir_z(x)$ determines a unique geodesic, let $\dir_z(x)$ also denote the geodesic ray starting at $x$ determined by the direction.

Then similarly to the argument of the proof of Proposition \ref{prop.homeo-image},
one sees that intersection of rays $\dir_z(x)$ is $z$.
This shows that inverse of $\Psi$ is continuous.
Combined with Lemma \ref{lem.horo}, we see that $\Psi$ is a homeomorphism onto its image.
\end{proof}

By Proposition \ref{prop.horo-homeo-riemann}, we see that the closure $\overline{\Psi(M)}$ is compact.
Hence we may define the following.
\begin{defi}
We call $\overline{\Psi(M)}$ the {\em horocoordinate compactification} and $\overline{\Psi(M})\setminus\Psi(M)$ the {\em horocoordinate boundary} of $M$ with respecet to $\Psi$.
\end{defi}
Then together with Proposition \ref{prop.BH-variation}, we have:
\begin{coro}\label{cor.int-formula}
Let $\xi = (\psi_\xi, \dir_\xi)\in\overline{\Psi(M)}$ and $\gamma:[0,d(b,x)]\rightarrow M$ denote the unique geodesic connecting $b$ and $x$ parametrized by arc length.
Then
$$\psi_\xi(x) = 	\int_0^{d(b,x)} \cos\theta_{\gamma(t)}\left(\dot\gamma(t), \dir_\xi(\gamma(t))\right)dt.$$
In particular, given $(\psi,\dir),(\psi', \dir')\in \overline{\Psi(M)}$ we have $(\psi,\dir) = (\psi', \dir')$ if and only if $\dir = \dir'$.
\end{coro}
\begin{proof}
	First, let us suppose $\xi = \Psi(z)$ for some $z\in M$.
	In CAT(0) spaces, the angles $\theta_t:=\theta_{\gamma(t)}\left(\dot\gamma(t), \dir_z(\gamma(t))\right)$ 
	is a continuous function of $t$, see \cite[Chapter II, Proposition 3.3]{BH}.
	Hence by Proposition \ref{prop.BH-variation}, we see that $\frac{d}{dt}d(\gamma(t),z) = \cos\theta_t$ is continuous and 
	$$d(x,z) - d(b,z) =  \int_0^{d(b,x)} \frac{d}{dt}d(\gamma(t),z)dt = \int_0^{d(b,x)} \cos\theta_t dt.$$
	
	Now if $\Psi(z_n)\rightarrow\xi$, then by definition, for any $x\in M$ and $v\in T_x^1(M)$, we have
	$$\cos\theta_{x}\left(v, \dir_{z_n}(x)\right)\rightarrow \cos\theta_{x}\left(v, \dir_\xi(x)\right).$$
	Then since $[0,d(b,x)]$ is compact, by the dominated convergence theorem,
	\begin{align*}
	\psi_\xi(x) = \lim_{n\rightarrow\infty}\psi_{z_n}(x) 
	&= 	
	\lim_{n\rightarrow\infty}\int_0^{d(b,x)} \cos\theta_{\gamma(t)}\left(\dot\gamma(t), \dir_{z_n}(\gamma(t))\right)dt\\
	&=
	\int_0^{d(b,x)} \lim_{n\rightarrow\infty}\cos\theta_{\gamma(t)}\left(\dot\gamma(t), \dir_{z_n}(\gamma(t))\right)dt\\
	&= \int_0^{d(b,x)} \cos\theta_{\gamma(t)}\left(\dot\gamma(t), \dir_\xi(\gamma(t))\right)dt.\\
	\end{align*}
	The last assertion follows from the integral formula and the fact $\psi(b) = \psi'(b) = 0$.
\end{proof}
\begin{rmk}
	As we have fixed the base point $b$, and $M$ is CAT(0), the geodesic $\gamma$ in the statement of Corollary \ref{cor.int-formula} is determined by $x$.
	Therefore the second coordinate of $\Psi(z)$ can be regarded as the data of derivatives of the horofunction in the first coordinate of $\Psi(z)$.
\end{rmk}

\begin{rmk}\label{rmk.1-lip}
	That horofunctions are $1$-Lipschitz is crucial when we compactify a given metric space.
	This follows from the triangle inequality.
	However, as we shall prove in \S \ref{sec.not-distance}, the function on the Teichm\"uller space given by the renormalized volume does not satisfy the triangle inequality.
	We remark here that as $\cos\theta\leq 1$, Corollary \ref{cor.int-formula} provides us an alternative proof of the $1$-Lipschitz property (compare with Theorem \ref{thm.integral} below).
\end{rmk}


\section{Preliminaries for Teichm\"uller theory}\label{sec.pre}
Our goal is to compactify the Teichm\"uller space by using the renormalized volume.
There will be an interplay of the theories of 2 and 3-dimensional hyperbolic geometry, complex analysis, and the geometric group theory.
In this section, we quickly review basic properties, mainly for the Teichm\"uller theory, that we need later.
Throughout the paper, we fix an orientable closed surface $S$ of genus greater than $1$.

\subsection{Beltrami differentials and quadratic differentials}
The Teichm\"uller space $\mathcal{T}(S)$ is the space of marked hyperbolic or complex structures on $S$.
Let $X\in\mathcal{T}(S)$. 
By abuse of notation, we regard $X$ as a Riemann surface or a hyperbolic surface depending on the context.
Let us first consider $X$ as a Riemann surface.
In Teichm\"uller theory, there are two important differentials, Beltrami differentials and quadratic differentials.
Let $T^{1,0}X$ and $T^{0,1}X$ denote the subspaces of the cotangent bundle which corresponds to holomorphic and anti-holomorphic part respectively.
A {\em Beltrami differential} is a section of $T^{0,1}X\otimes (T^{1,0}X)^{*}$ which is locally expressed by $\beta(z)d\bar z/dz$.
A {\em quadratic differential} is a section of $T^{1,0}X\otimes T^{1,0}X$ whose local expression is $q(z)dz^{2}$.
If moreover, $q(z)$ is holomorphic on each local chart, it is called a holomorphic quadratic differential.
The space of holomorphic quadratic differentials on $X$ is denoted by $\QD(X)$.
By the Riemann-Roch theorem, $\QD(X)$ is isomorphic to $\mathbb{C}^{3g-3}$.
Hence the space of all holomorphic quadratic differentials on $S$ defines a vector bundle over $\T$, that we denote by $\QD(S)$.

Quadratic differentials arise in several natural ways in the theory of Teichm\"uller spaces, see e.g. \cite{Gupta} for details.
In this paper, quadratic differentials that appear as {\em Schwarzian derivatives} play an important role, 
whose definition we postpone until \S \ref{sec. QF and S}.
Here we will recall so-called Hubbard-Masur differentials \cite{HM}.
A measured foliation on $S$ is a singular foliation with a transverse measure.
Let $\mathcal{MF}(S)$ denote the space of measured foliations on $S$, 
see \S \ref{sec. compact T} below for the topology on $\mathcal{MF}(S)$ 
and its relation to simple closed curves and so on.
Given a quadratic differential $q$ locally expressed as $q(z)dz^{2}$ on $X$,
 its horizontal directions is the set of directions $v\in T_{z}X$ defined by $q(z)v^{2}\in\mathbb{R}_{>0}$.
The horizontal direction equipped with the transversal measure defined for any transversal arc $\alpha$ by
$$\int_{\alpha}|\mathrm{Im}~ q(z)^{1/2}dz|$$
determines a measured foliation which is called the {\em horizontal foliation} of $q$.
Let $h_{X}(q)$ denote the horizontal foliation of $q$.
Thus we get a map
$h_{X}:\QD(X) \rightarrow \mathcal{MF}(S)$.
Hubbard-Masur differential is defined as the converse of this map:
\begin{thm}[\cite{HM}]\label{thm.HM}
The map $h_{X}:\QD(X) \rightarrow \mathcal{MF}(S)$ is a homeomorphism.
\end{thm}

\begin{notation*}
By Theorem \ref{thm.HM}, each $F\in\mathcal{MF}(S)$ corresponds to a quadratic differential 
$q_{F}(X):=h_{X}^{-1}(F)$. This foliation $q_{F}(X)$ is called {\em the Hubbard-Masur quadratic differential}
associated to $F$ on $X$.
\end{notation*}

There are several types of natural norms on $\QD(X)$.
First, we recall the $L^{1}$-norm, the area of the quadratic differential, 
and $L^2$-norm.
$$||q||_{1}:=\int_{X}|q| 
\text{~~, and }~~ ||q||_2:=\left(\int_X\frac{|q|^2}{\rho^2}\right)^{1/2},$$
where $\rho |dz|$ is the hyperbolic metric determined by $X$.
We also need so-called the $L^{\infty}$-norm, or the sup norm:
$$||q||_{\infty}:=\sup_{z\in X}\frac{|q(z)|}{\rho^2(z)}.$$

Let $\beta$ be a Beltrami differential and $q\in \QD(X)$.
The product of $\beta$ and $q$ gives a section in $T^{0,1}X\otimes T^{1,0}X$.
Hence there is a natural pairing of $q$ and $\beta$.
$$\langle q,\beta\rangle:=\int_{X} q\beta.$$
The real part $\mathrm{Re}\langle q,\beta\rangle$ turns out to be very important for the study of the renormalized volume (see Theorem \ref{thm.dRvol} below).
By the pairing, we define the $L^{\infty}$-norm of Beltrami differentials:
$$||\beta||_{\infty}:=\sup_{||q||_{1} = 1}\langle q,\beta\rangle.$$
Then let
$L^{\infty}(X):=\{\beta\mid ||\beta||_{\infty}<\infty\}$ be 
the space of Beltrami differentials with bounded $L^{\infty}$-norm.
It has a subspace
$$K := \{\beta\in L^{\infty}(X)\mid \langle q,\beta\rangle = 0 \text{ for any }q\in \QD(X)\}.$$
The quotient $L^{\infty}(X)/K$ can be identified with the tangent space of $\mathcal{T}(S)$ at $X$.
See e.g. \cite[Chapter 11]{FaM}.
For a differentiable path $\sigma$ in $\T$ with $\sigma(t) = X$, we denote by $\dot\sigma(t)$ 
the corresponding Beltrami differential in $L^{\infty}(X)/K$.

\subsection{Extremal length and Kerckhoff's formula}
A closed curve on $S$ is called {\em simple} if it is homotopic to a curve without self-intersections.
A simple closed curve is called {\em essential} if it does not homotopic to a point.
Let $\mathcal{S}$ be the set of homotopy classes of essential simple closed curves.
The extremal length is a complex analogue of hyperbolic length.
Note that we may regard $\mathcal{S}$ as a subset of $\mathcal{MF}(S)$ by giving intersection numbers as transverse measures.
Given $F\in\mathcal{MF}(S)$, the {\em extremal length} $\ext_{X}(F)$ of $F$ on $X$ is
$$\ext_{X}(F):=||q_{F}(X)||_{1},$$
where $q_{F}(X)$ is the Hubbard-Masur differential defined above.
On $\mathcal{T}(S)$, the Teichm\"uller distance is defined as the measurement of the deformation of complex structures.
Kerckhoff's formula \cite{Kerckhoff} gives an alternative description of the Teichm\"uller distance in terms of the extremal length:
\begin{equation}\label{eq.Kerckhoff}
d_{\mathcal{T}}(X,Y) = \frac{1}{2}\log\sup_{\alpha\in\mathcal{S}}\frac{\ext_{Y}(\alpha)}{\ext_{X}(\alpha)}.
\end{equation}
We may regard this formula as the definition of the Teichm\"uller distance.
\subsection{Quasi-Fuchsian manifolds and Schwarzian derivatives}\label{sec. QF and S}
A {\em Kleinian group} is a discrete subgroup of $\PSL$.
In this paper we only consider surface Kleinian groups (i.e. Kleinian groups that are isomorphic to $\pi_{1}(S)$).
Let $\Gamma$  be a surface Kleinian group.
$\Gamma$ acts on the hyperbolic 3-space $\mathbb{H}^{3}$, and {\em the limit set} $\Lambda(\Gamma)\subset\partial\HH$ of $\Gamma$ is the set of accumulation points of any orbit of $\Gamma$ in $\HH$.
The complement $\partial\HH\setminus\Lambda(\Gamma)$ is called {\em the domain of discontinuity}, denoted $\Omega(\Gamma)$.
$\Gamma$ is called {\em quasi-Fuchsian} if $\Lambda(\Gamma)$ is a Jordan curve and hence $\Omega(\Gamma)$ has exactly two components.
In this paper, we denote those components of domain of discontinuity by $\Omega_{+}(\Gamma)$ and $\Omega_{-}(\Gamma)$ and call the {\em top component} and the {\em bottom component} respectively.
Let $$\AH:= \left\{\rho\in\mathrm{Hom}(\pi_{1}(S),\PSL)\mid \rho\text{ is discrete and faithful}\right\}/\text{conjugation}$$ denote the space of complete hyperbolic metrics on the surface group with the topology of representations.
The space $\mathrm{QF}(S)$ of quasi-Fuchsian surface Kleinian groups is a subset of $\AH$.
As $\partial\HH\cong\hat{\mathbb{C}}$, both $\Omega_{+}(\Gamma)/\Gamma$ and $\Omega_{-}(\Gamma)/\Gamma$ have natural complex structures.
By the Bers simultaneous uniformization theorem, this correspondence is a parametrization i.e. the map
$$\qf:\TT\rightarrow\QF(\subset\AH)$$
is a homeomorphism. 
Namely, given any $(X,Y) \in \TT$, $\qf(X,Y)$ is quasi-Fuchsian and its top component and bottom component have complex structures $X$ and $Y$ respectively.
We denote by $M(X,Y)$ the quotient $\HH/\qf(X,Y)$ of the quasi-Fuchsian group $\qf(X,Y)$.
The quotient $M(X,Y)$ is called a {\em quasi-Fuchsian manifold}.
Note that the space $\mathcal{T}(S)\times\mathcal{T}(\bar S) $ where $\bar S$ is $S$ with the orientation reversed parametrizes $\QF$ more naturally.
But as $\mathcal{T}(S)\cong\mathcal{T}(\bar S)$ and for later convenience, we adopt $\TT$ as a parametrization space of $\QF$.
With this parametrization $\qf(X,X)$ corresponds to the Fuchsian representation corresponding to $X\in\T$.

We will now discuss associated Schwarzian derivatives.
Let $U\subset\hat{\mathbb{C}}$ be a connected open set.
Given a holomorphic map $f:U\rightarrow\hat{\mathbb{C}}$, 
the {\em Schwarzian derivative} $S(f)$ of $f$ is a holomorphic quadratic differential on $U$ defined by 
$$S(f):=\left(\left(\frac{f''}{f'}\right)'-\frac{1}{2}\left(\frac{f''}{f'}\right)^{2}\right)dz^{2}.$$
To define the Schwarzian derivative associated to quasi-Fuchsian manifolds,
 we first recall complex projective structures.
A {\em complex projective structure} on $S$ is a complex structure with every transition map between local charts is a restriction of 
the action of  some element in $\PSL$.
Let $P(S)$ denote the space of complex projective structures.
Since the action of $\PSL$ is holomorphic, we have a natural forgetful map
$$P(S)\rightarrow\T.$$
Let $P(X)$ denote the preimage of $X\in\T$ in $P(S)$.
Suppose we have two projective structures $Z_{1}, Z_{2}\in P(X)$.
For sufficiently small open set $U\subset S$ so that there are projective coordinate charts $z_{i}:U\rightarrow\hat{\mathbb{C}}$ of $Z_{i}$ for $i=1,2$, we have a holomorphic quadratic differential
$z_{1}^{*}S(z_{2}\circ z_{1}^{-1})$ on $U\subset X$.
Since the Schwarzian derivative of any M\"obius transformation vanishes, 
we have a well-defined holomorphic quadratic differential, denoted $Z_{2}-Z_{1}$ in $\QD(X)$ by covering $X$ with such open sets.

The restriction of any quasi-Fuchsian representatives of $\pi_{1}(S)$ on a component of domain of discontinuity gives a complex projective structure.
Let $P_{Y}(X)\in P(X)$ denote the complex projective structure given as a quotient $\Omega_{+}/\qf(X,Y)$ of the top component $\Omega_{+}$ of $\qf(X,Y)$.

\begin{notation*}
We denote by $q_{Y}(X):=P_{Y}(X)-P_{X}(X)\in \QD(X)$
the holomorphic quadratic differential obtained as the Schwarzian derivative.
\end{notation*}
Note that $q_{Y}(X)$ is determined only by $\qf(X,Y)$.
We now recall the famous Nehari's inequality on the Schwarzian derivatives:
\begin{thm}[\cite{Nehari}]\label{thm.Nehari}
Let $\qf(X,Y)$ be a quasi-Fuchsian surface Kleinian group.
Then we have 
$$||q_{Y}(X)||_{\infty}\leq \frac{3}{2}.$$
\end{thm}
The Nehari's inequality can be re-interpreted as follows.
Recall that $g$ is the genus of $S$.

\begin{coro}[\cite{Nehari}, see also \cite{Schlenker-MRL, Kojima-McShane}]\label{lem.Nehari}
Let $\qf(X,Y)$ be a quasi-Fuchsian surface Kleinian group.
Then we have 
$$
||q_{Y}(X)||_{1}\leq \ConT \text{, and } ||q_Y(X)||_2\leq \Con.$$
\end{coro}

\subsection{Renormalized volume}
Although quasi-Fuchsian manifolds have infinite hyperbolic volume, there is a notion called the renormalized volume which is finite for any quasi-Fuchsian manifold.
The idea of the renormalized volume comes from Graham-Witten \cite{GW} and it is studied by several authors for hyperbolic $3$-manifolds (see e.g. \cite{BBB, BBB2,BBP, BC, GMR, Kojima-McShane, KS, Schlenker-MRL, Schlenker}).
\begin{notation*}
Via renormalized volume, we get a function
$$\VR:\TT\rightarrow \mathbb{R}$$
defined so that $\VR(X,Y)$ is the renormalized volume of the quasi-Fuchsian manifold $M(X,Y)$.
\end{notation*}

The formal definition involves the mean curvature, Epstein surfaces, and other notions that we do not need for the discussion in this paper.
We refer \cite{Schlenker, Schlenker-MRL, KS, BBB, BC,  Kojima-McShane} and references therein for more details.
Instead of giving the original definition, we adopt the following formula of the first variation of the renormalized volume as the definition.
\begin{thm}[{\cite[Lemma 2.4]{Kojima-McShane}, \cite[Corollary 3.13]{Schlenker}}]\label{thm.dRvol}
For any $Y\in\T$, $\VR(\cdot, Y)$ is differentiable on $\T$.
If $\sigma:[-1, 1]\rightarrow\T$ is a differentiable path,
$$\left.\frac{d}{dt}\right|_{t=0}\VR({\sigma(t)}, Y) = -\mathrm{Re}\langle q_{Y}(\sigma(0)),\dot\sigma(0)\rangle.$$
\end{thm}

\section{Known compactifications of Teichm\"uller space}\label{sec.TTBcpt}
In this section we recall several known natural compactifications of the Teichm\"uller space $\T$.
Similarly to the boundary of the hyperbolic space, the Teichm\"uller space is compactified with boundaries ``at infinity''.
Each compactification captures different kinds of asymptotic behaviors in the Teichm\"uller space.
Recall that the mapping class group $\mcg$ is the group of isotopy classes of homeomorphisms on $S$.
The group $\mcg$ acts on $\T$ by the change of markings.
The action sometimes extends to the boundaries, and sometimes does not.
\subsection{Thurston compactification and Gardiner-Masur compactification}\label{sec. compact T}
Let $\mathcal{S}$ denote the space of homotopy classes of essential simple closed curves.
If we have any ``length'' function determined by a point $X\in\mathcal{T}(S)$, we get a point in 
$\mathbb{R}_{\geq 0}^{\mathcal{S}}$.
By giving the geometric intersection number $i(\cdot,\cdot):\mathcal{S}\times\mathcal{S}\rightarrow \Z$, 
the set $\mathcal{S}\times\mathbb{R}_{>0}$, that is $\mathcal{S}$ together with positive real weights, embeds into $\mathbb{R}_{\geq 0}^{\mathcal{S}}$ by sending $(\alpha,t)\in\mathcal{S}\times\R_{>0}$ to $t\cdot i(\alpha,\cdot)\in\R^\mathcal{S}_{\geq 0}$.

The space $\mf$ of measured foliations is obtained as the completion of 
$\mathcal{S}\times\mathbb{R}_{>0}$ in $\mathbb{R}_{\geq 0}^{\mathcal{S}}$.
The positive real numbers $\mathbb{R}_{>0}$ acts on $\mathbb{R}_{\geq 0}^{\mathcal{S}}$ by multiplication and its quotient is denoted by $P\mathbb{R}_{\geq 0}^{\mathcal{S}}$.
Thurston (c.f. \cite{FLP}) proved that if we use the hyperbolic length and define the map $i_{\mathrm{Th}}:\mathcal{T}(S)\rightarrow P\mathbb{R}_{\geq 0}^{\mathcal{S}}$ by $i_\mathrm{Th}(X)(\alpha) = \ell_X(\alpha)$, then $i_\mathrm{Th}$ is an embedding.
The closure $\overline{i_{\mathrm{Th}}(\mathcal{T}(S))}$ is called the {\em Thurston compactification}.
Let $\partial_{\mathrm{Th}}\mathcal{T}(S):= \overline{i_{\mathrm{Th}}(\mathcal{T}(S))}\setminus i_{\mathrm{Th}}(\mathcal{T}(S))$ denote the added boundary called the {\em Thurston boundary}.
The image of $\mathcal{MF}(S)$ in $P\mathbb{R}_{\geq 0}^{\mathcal{S}}$ is called the space of projective measured foliations and denoted $\mathcal{PMF}(S)$.
Thurston \cite{FLP} proved that in $P\mathbb{R}_{\geq 0}^{\mathcal{S}}$ 
the closure $\overline{i_{\mathrm{Th}}(\mathcal{T}(S))}$ is identified with $\mathcal{T}(S)\cup\mathcal{PMF}(S)$.

Similarly, Gardiner-Masur \cite{GM} considered (the square root of) the extremal length to define
$i_{\mathrm{GM}}:\mathcal{T}(S) \rightarrow P\mathbb{R}_{\geq 0}^{\mathcal{S}}$.
The closure $\overline{i_{\mathrm{GM}}(\mathcal{T}(S))}$ in $P\mathbb{R}_{\geq 0}^{\mathcal{S}}$ is called the {\em Gardiner-Masur compactification} and the boundary 
$\partial_{\mathrm{GM}}\mathcal{T}(S):= \overline{i_{\mathrm{GM}}(\mathcal{T}(S))}\setminus i_{\mathrm{GM}}(\mathcal{T}(S))$ is called the {\em Gardiner-Masur boundary.}
Gardiner-Masur also proved that $\partial_{\mathrm{GM}}\mathcal{T}(S)$ strictly contains $\mathcal{PMF}(S)$.

Similarly to Kerckhoff's formula of the Teichm\"uller distance, the {\em Thurston distance}, denoted $d_{\mathrm{Th}}(\cdot, \cdot)$ is characterized as
\begin{equation}\label{eq.dTh}
d_{\mathrm{Th}}(X,Y):=\log\sup_{\alpha\in\mathcal{S}}\frac{\ell_{Y}(\alpha)}{\ell_{X}(\alpha)},
\end{equation}
where $\ell_{X}(\alpha)$ is the hyperbolic length of the simple closed curve $\alpha$
 with respect to the hyperbolic metric $X$ (\cite{Thurston}).
 Both the Thurston boundary and the Gardiner-Masur boundary are identified with horoboundaries.
\begin{thm}[Walsh \cite{Wal}]\label{thm.horoWal}
The horoboundary with respect to the Thurston distance is homeomorphic to the Thurston boundary 
$\partial_{\mathrm{Th}}\mathcal{T}(S)$.
\end{thm}
\begin{rmk}
The Thurston distance is asymmetric.
Hence we need more care to define horoboundaries, see \cite{Wal} for the details.
\end{rmk}

Similarly if one uses the Teichm\"uller distance, one gets:
\begin{thm}[Liu-Su \cite{LS}]\label{thm.horoGM}
The horoboundary with respect to the Teichm\"uller distance is homeomorphic to the Gardiner-Masur boundary
$\partial_{\mathrm{GM}}\mathcal{T}(S)$.
\end{thm}

Thus we see that horoboundaries give natural boundaries for the distances.

\subsection{Bers compactification and ending laminations}\label{sec.Bers-end}
As discussed in \S \ref{sec. QF and S}, the space $\QF$ of quasi-Fuchsian groups is parametrized by 
$\TT$.

Let us fix $Y\in\mathcal{T}(S)$ and consider $\qf(\cdot, Y):\T\rightarrow\QF$.
By considering the Schwarzian derivative $q_{X}(Y)$, Bers considered the map $b_{Y}:\mathcal{T}(S)\rightarrow \QD(Y)$.

Let $\QD_{B}(Y):=\{q\in\QD(Y)\mid ||q||_{\infty}\leq 3/2\}$.
Nehari's inequality (Theorem \ref{thm.Nehari}) implies that 
the closure $\overline{b_{Y}(\T)}$ is contained in $\QD_{B}(Y)$.
Thus Bers showed that the closure $\overline{b_{Y}(\mathcal{T}(S))}$ is compact,
and it is called the {\em Bers compactification} with base point $Y$.
The boundary $\partial^{Y} _{B}\mathcal{T}(S)$ is called the {\em Bers boundary}.
Note that Bers boundaries depend on the base point $Y$, and the action of the mapping class group does {\em not} extend continuously (see Kerckhoff-Thurston \cite{KT}).

To characterize points in $\partial_{B}^{Y}\mathcal{T}(S)$, we recall the so-called ending laminations.
As the detailed discussion is not necessary in this paper, the exposition here is very brief, see e.g. \cite{BCM,Ohshika} and references therein for more details.
Let us fix a hyperbolic metric on $S$ for a moment.
A {\em geodesic lamination} is a closed subset of $S$ consisting of simple geodesics,
each of which is called a {\em leaf}.
A geodesic lamination $\lambda$ is called {\em minimal} if every leaf is dense in $\lambda$.
Let $\mathcal{GL}(S)$ denote the space of geodesic lamination with Hausdorff topology.
The set of simple closed curves $\mathcal{S}\subset\mathcal{GL}(S)$ is known to be dense.
The space $\mathcal{GL}(S)$ is known to be independent of the choice of the hyperbolic structure on $S$, see e.g. \cite{CB}.
Note as the topologies of $\AH$ and $\QD(S)$ are compatible, we may regard each  $\xi\in\partial_{\mathrm{B}}^{Y}\mathcal{T}(S)$
as a point in $\AH$.
Let $\rho_{\xi}:\pi_{1}(S)\rightarrow\PSL$ be an arbitrarily chosen representative corresponding to $\xi$.
The properties we are discussing below are invariant under conjugation and therefore independent of the choice of  representatives.

A point $\xi\in\partial_{\mathrm{B}}^{Y}\mathcal{T}(S)$ is said to have {\em accidental parabolics} 
if $\rho_{\xi}(\pi_{1}(S))$ have some parabolic elements.
If a point $\xi\in\partial_{\mathrm{B}}^{Y}\mathcal{T}(S)$ does not have accidental parabolics, 
then it is called {\em singly degenerate}.
If $\xi$ is singly degenerate, the corresponding manifold $\HH/\rho_{\xi}(\pi_{1}S)$ has two ends, one of which corresponds to the fixed $Y\in\mathcal{T}(S)$.
For a singly degenerate $\xi$, there is a sequence $\{\alpha_{i}\}\subset\mathcal{S}$ of simple closed curves whose geodesic representative eventually enters any neighborhood of the end that is not the $Y$-side.
It turns out all such sequences have a unique limit in $\mathcal{GL}(S)$.
The limit is called the {\em ending lamination}, and denoted $E(\xi)$.
(A special case of) the ending lamination theorem says the following:
\begin{thm}[Ending lamination theorem {\cite{BCM}}]\label{thm.elt}
If $\xi$ is singly degenerate, then the isometry type of $\HH/\rho_{\xi}(\pi_{1}S)$ is determined by $Y$ and $E(\xi)$.
\end{thm}
Finally, we discuss ``intersections'' of $\partial_{\mathrm{Th}}\mathcal{T}(S)$, $\partial_{\mathrm{GM}}\mathcal{T}(S)$ and $\partial_{\mathrm{B}}^{Y}\mathcal{T}(S)$.
A {\em measured lamination} is a geodesic lamination equipped with a transverse measure.
Let $\mathcal{ML}(S)$ denote the space of measured laminations.
There is a canonical one-to-one correspondence between $\mathcal{ML}(S)$ and the space of measured foliations $\mathcal{MF}(S)$ see e.g. \cite{CB, FLP}.
Hence we may identify $\mathcal{ML}(S)$ with $\mathcal{MF}(S)$.
A measured lamination (or foliation) is called {\em uniquely ergodic} 
if every transverse measure on the underlying lamination (or foliation) are related by the multiple of a positive real number.
Let $\mathcal{UE}(S)$ denote the set of uniquely ergodic laminations or foliations, 
which by abuse of notation we regard as a subset of one of $(\mathcal{P})\mathcal{ML}(S)$ or $(\mathcal{P})\mathcal{MF}(S)$ depending on the context.
By definition, the measure forgetting map $\mathcal{PML}(S)\rightarrow\mathcal{GL}(S)$ is one-to-one on $\mathcal{UE}(S)$.
Thus we may also regard $\mathcal{UE}(S)\subset\mathcal{GL}(S)$.
Let us recall the work of Brock.
\begin{thm}[{\cite[Theorem 6.1]{Brock}}]\label{thm.Brock}
Let $\{X_{n}\}\subset\T$ be a sequence that converges as $n\rightarrow\infty$ to $\mu$ in the Thurston boundary $\pmf = \partial_{\mathrm{Th}}\T$.
Then for any limit $\xi\in\partial_{B}^{Y}\T$ of $X_{n}$, 
the support lamination $|\mu|$ of $\mu$ is a sublamination of $E(\xi)$.
\end{thm}
We combine Theorem \ref{thm.Brock} with the following observation of Masur.
\begin{lem}[{\cite[Lemma 2]{Masur}}]\label{lem.ue-inter}
	Let $\mu$ and $\lambda$ be geodesic laminations.
	Suppose further that $\lambda$ is uniquely ergodic.
	Then $i(\mu,\lambda) = 0$ if and only if $\mu=\lambda$.
\end{lem}
Hence if $\mu$ is uniquely ergodic, then any lamination containing $\mu$ must coincide with $\mu$.
Therefore Theorem \ref{thm.elt} shows the following.
\begin{coro}\label{coro.elt}
Let $\mu\in\mathcal{UE}(S)$.
If $X_{n}\rightarrow\mu$ in the Thurston boundary $\partial_{\mathrm{Th}}\T$,
then the limit $\xi$ of $X_{n}$ exists in the Bers boundary $\partial_{B}^{Y}\T$ for any $Y$,  and $E(\xi) = |\mu|$.
\end{coro}


\section{Compactification of Teichm\"uller space via Weil-Petersson metric}\label{sec.WP}
Before discussing a compactification via renormalized volume, 
let us apply our strategy to the Weil-Petersson (WP) metric on the Teichm\"uller space.
For properties of the Weil-Petersson metric discussed here, see e.g. Wolpert's book \cite{Wolpert-book} and survey article \cite{Wolpert-survey}.
Let 
$$\langle\varphi,\psi\rangle_{\mathrm{WP}} =\int_X\frac{\varphi\overline\psi}{\rho^2}$$
be the $L^2$ inner product on the cotangent space $QD(X) = T_X^*\mathcal{T}(S)$, where
$\rho(z)|dz|$ is the hyperbolic metric on $X$.
The real part of the dual of $\langle\cdot,\cdot\rangle_{\mathrm{WP}}$ defines a Riemannian metric on $\mathcal{T}(S)$, which is called the WP metric.
The WP metric is not complete, and let $\hatT$ denote the completion which is characterized as the augmented Teichm\"uller space by Masur \cite{Masur}.
Although the Teichm\"uller space $\T$ equipped with WP metric is not proper, it is a CAT(0) space which in particular is uniquely geodesic (see \cite[Theorem 13 and Theorem 14]{Wolpert-survey}, \cite{Yamada}).
One useful feature of the WP metric is the following.
\begin{lem}[{\cite[Theorem 5]{Wolpert-survey}}]\label{lem.wp-exp}
	The WP exponential map from any base point is a diffeomorphism from 
	its open domain onto the Teichm\"uller space.
\end{lem}

These features of WP metric let us apply arguments in \S \ref{sec.Angles}, and obtain a compactification of $\T$ via the WP metric as follows.
Let $\dwp$ denote the distance function with respect to the WP metric and 
$\omega_Z:\mathcal{T}(S)\rightarrow \R$ denote the horofunction at $Z$ i.e. $\omega_Z(X) = \dwp(X,Z)-\dwp(b,Z)$, where $b\in\mathcal T(S)$ is a base point which we fix throughout the section.
For the consistency with the case of renormalized volume, let us 
consider the space of holomorphic quadratic differentials as the space of directions.
Now we define the space $\LQ$ which is the space of sections of a bundle over $\T$.
\begin{defi}\label{defi.LQ}
Let $C := \Con$.
Then we define $$\LQ:=\prod_{X\in\T}\left\{[-C\dwp(b,X),C\dwp(b,X)]\times \QD_{B}(X)\right\},$$
(LQ stands for \underline Lipschitz and \underline Quadratic differential).
Furthermore by the notation $(\xi,q)\in\LQ$, we mean a point given by 
$\xi:\T\rightarrow\mathbb{R}$ and $q:\T\rightarrow\QD(S)$, where 
$\xi(X)\in[-\dwp(b,X),\dwp(b,X)]$ and $q(X)\in\QD_{B}(X)$.
We equip $\LQ$ with the topology of point-wise convergence, or equivalently the product topology.
\end{defi}
\begin{rmk}
	The space $\LQ$ is designed for the renormalized volume which we discuss in \S \ref{sec.main}.
	For this reason, we expand the interval by the constant $C$ and utilize the space $QD_B(X)$, the target space of the Bers embedding.
	Although $\LQ$ is larger than necessary, it works for the WP metric as well.
\end{rmk}
We summarize the property of the space $\LQ$.
\begin{prop}[c.f. Proposition \ref{prop.Lip-compact}]\label{prop.LQ}
The space $\LQ$ is a compact, Hausdorff, and second countable (hence metrizable) space.
\end{prop}
\begin{proof}
As $[-C\dwp(b,X),C\dwp(b,X)]\times \QD_{B}(X)$ is compact, the whole space $\LQ$ is compact by Tychonoff's theorem.
That $\LQ$ is Hausdorff follows because $[-C\dwp(b,X),C\dwp(b,X)]\times \QD_{B}(X)$ is Hausdorff.
Since the Teichm\"uller space $\T$ is separable and $[-C\dwp(b,X),C\dwp(b,X)]\times \QD_{B}(X)$ are second countable, 
$\LQ$ is also second countable.
\end{proof}

Since $(\mathcal T(S), \dwp)$ is uniquely geodesic, given any distinct two points $X,Y\in\mathcal{T}(S)$, there is a unique direction $\overline \dir_Y(X)\in T^1_X\mathcal T(S)$ from $X$ toward $Y$.
Let $\bar D_Y(X)\in \QD(X)\cong T^*_X\mathcal T(S)$ denote the dual of $\overline \dir_Y(X)$ of norm $||\bar D_Y(X)||_\infty = 1$ (in fact, any norm would work.  
We chose $L^\infty$ norm merely for the consistency with Bers embeddings which we discuss in \S \ref{sec.main}).
Let $f(x) = (e^x-1)/(e^x+1)$ as in \S \ref{sec.Angles}. 
Then we define $D_Y(X):=f(\dwp(X,Y))\cdot\bar D_Y(X)$.
On $\LQ$, the mapping class group $\mcg$ acts by $g\cdot\omega(X) = \omega(g^{-1}X) - \omega(g^{-1}b)$ for horofunctions $\omega$ (see Lemma \ref{lem.action}), and $g\cdot D_Y(X) := D_{gY}(gX)$ for $D_Y(X)$'s.
\begin{thm}\label{thm.wp-cpt}
	The embedding $\mathcal{WP}:\mathcal{T}(S)\rightarrow\LQ$ defined by
	$$\mathcal{WP}(Z) = \left(\omega_Z(X), D_Z(X)\right)_{X\in\mathcal T(S)}$$
	is a homeomorphism onto its image.
	Its closure $\overline{\mathcal{WP}(\mathcal{T}(S))}$ is compact, and the action of the mapping class group extends continuously on $\overline{\mathcal{WP}(\mathcal{T}(S))}$.
\end{thm}
\begin{proof}
	Since $D_Y(X)$'s work as the directions, we see that $\mathcal{WP}$ is a homeomorphism onto its image by Proposition \ref{prop.horo-homeo-riemann}.
	Since the space $\LQ$ is compact,
	the closure $\overline{\mathcal{WP}(\mathcal{T}(S))}$ is compact.
	The action of the mapping class group is by isometries with respect to WP metrics on $\T$.
	Hence by Lemma \ref{lem.action}, we see that the action extends continuously to the space of horofunctions.
	Finally, given $g\in\mcg$, the natural action $g(D_Z(X)) = D_{gZ}(gX)$ is the change of markings, and extends naturally and continuously to the closure $\overline{\mathcal{WP}(\mathcal{T}(S))}$.
\end{proof}
As the WP metric is Riemannian, we may naturally consider angles.
Then Corollary \ref{cor.int-formula} implies the following.
\begin{coro}\label{cor.int-formula-WP}
Let $(\omega, D)\in\overline{\mathcal{WP}(\mathcal{T}(S))}$ and $\gamma:[0,\dwp(b,X)]\rightarrow \T$ denote the WP geodesic connecting $b$ and $X$ parametrized by arc length.
Then
$$\omega(X) = 	\int_0^{\dwp(b,X)} \cos\theta_{\gamma(t)}\left(\dot\gamma(t), D(\gamma(t))\right)dt.$$
In particular, it follows that $(\omega,D) = (\omega',D')$ if and only if $D = D'$.
\end{coro}
\subsection{(In)finite horofunctions and the visual sphere}
Brock \cite{Brock2} has observed that the visual spheres with respect to the WP metric have properties that they
\begin{itemize}
	\item do depend on the base points, and the action of the mapping class group does not extend continuously, and
	\item contain finite points as dense subsets
\end{itemize}
where finite points here means that the associated geodesic rays have bounded lengths (and hence is incomplete).
The metric completion $\hatT$ turns out to be the union of $\mathcal T(S)$ and endpoints of those finite rays \cite{Masur-augmented, Wolpert-book}.
Let us recall a property of geodesics in $\hatT$ 
which leads to the fact that $\hatT$ is a CAT(0) space.

\begin{lem}[{c.f. \cite[Section 7]{Wolpert-survey}, \cite[Theorem 2.1]{Brock2}}]\label{lem.non-refraction}
	Let $(X_n, Y_n)\in\TT$ be a sequence of pairs with the limit $(X_\infty, Y_\infty)\in\hatT\times\hatT$.
	We denote by $\gamma_n$ the unique WP geodesic connecting $X_n$ and $Y_n$.
	Then $\gamma_n$ converges to the unique WP-geodesic $\gamma_\infty$ connecting $X_\infty$ and $Y_\infty$.
\end{lem}

Let $\pwp\T:=\overline{\mathcal{WP}(\mathcal T(S))}\setminus\mathcal T(S)$.
As in \S \ref{sec.finite-horo}, let
\begin{align*}
	\pwp^\mathrm{fin}\T &:= \{(\omega, D)\in\pwp\mathcal T(S)\mid \inf_{X\in\T}\omega(X)>-\infty\} \text{ and, }\\
	\pwp^\infty\T &:= \pwp\T\setminus\pwp^\mathrm{fin}\T.
\end{align*}
Let $(\omega,D)\in\pwp\T$.
In order to discuss properties of elements in $\pwp^\mathrm{fin}\T$ and $\pwp^\infty\T$,
we need to consider the WP geodesic ray determined by the direction $D(X)$.
By abuse of the notation, we also denote by $D(X)$ the geodesic ray starting at $X$ with direction $D(X)$.

\begin{thm}\label{thm.wp-finite}There is an injection $\hatT\setminus\T\rightarrow\pwp^\mathrm{fin}\T$.
	Conversely, if $(\omega,D)\in\pwp^\mathrm{fin}\T$ then for every $X\in\T$, 
	the direction $D(X)$ corresponds to a finite geodesic ray with endpoint in $\hatT\setminus\T$.
	Moreover, in general, the endpoint of the ray $D(X)\in\hatT\setminus\T$ depends on $X$.
\end{thm}
\begin{proof}
	Let $\{Z_n\}$ be a Cauchy sequence converging to $Z_\infty\in\hatT\setminus\T$.
	Then by the triangle inequality, we have $\dwp(X,Z_n)\rightarrow\dwp(X,Z_\infty)$ for any $X\in\T$.
	Hence the horofunction $\omega_{Z_\infty}$ at $Z_\infty$ is well-defined and we have $\omega_{Z_n}\rightarrow\omega_{Z_\infty}$ without taking subsequences.
	Furthermore, by Lemma \ref{lem.non-refraction},
	the direction $D_{Z_\infty}(X)$ from any $X$ to $Z_\infty$ is uniquely determined and coincides with $\lim_{n\rightarrow\infty} D_{Z_n}(X)$.
	As $\dwp(b,Z_\infty)<\infty$, we see that $(\omega_{Z_\infty}, D_{Z_\infty})\in\pwp^\mathrm{fin}\T$.
	Thus we get a map $\hatT\setminus\T\rightarrow\pwp^\mathrm{fin}\T$ and this is an injection since $\dwp$ extends to the completion $\hatT$.

	Conversely, let $(\omega, D)\in \pwp^\mathrm{fin}\T$.
	Let $\{Y_n\}\subset\T$ be a sequence so that $\mathcal{WP}(Y_n)\rightarrow (\omega, D)$. 
	Now suppose contrary that the ray $D(X)$ were infinite for some $X\in\T$.
	In this case, if $\sup\dwp(X,Y_n)$ is finite, then $Y_n$ must converge to an interior point of the ray $D(X)$, which contradicts the assumption $(\omega,D)\in\pwp\T$ by Theorem \ref{thm.wp-cpt}.
	If $\dwp(X,Y_n)\rightarrow\infty$, then for any $s\in\R_{>0}$, 
	we have a point $s_n$ on the ray $D_{Y_n}(X)$ with $\dwp(X,s_n)=s$.
	As $D_{Y_n}(X)\rightarrow D(X)$ and the exponential map is diffeomorphic (Lemma \ref{lem.wp-exp}), the sequence $s_n$ must converge to a point $s_\infty$ on the ray $D(X)$.
	Now notice that the triangle inequality implies
	\begin{align*}
		\omega_{Y_n}(s_n) &= \dwp(s_n, Y_n) - \dwp(b ,Y_n)\\
		&\leq \dwp(s_n, Y_n) - \dwp(X , Y_n) + \dwp(b,X)= -s + \dwp(b,X).
	\end{align*}

	The diagonal argument shows that $\omega(s_\infty) \leq -s + \dwp(b,X)$.
	Since $s\in\R_{>0}$ was arbitrary, it follows $(\omega,D)\in\pwp^{\infty}\T$, which contradicts our assumption.
	Thus we see that $D(X)$ is a finite ray for any $X\in\T$.
	
	Finally, it is the work of Brock \cite[Theorem 1.7]{Brock2} 
	that in general the limit point in $\hatT\setminus\T$ of $D(X)$ depend on $X$.
	This happens for example when $Y_n$ are obtained by iterating Dehn twists and the sequence $\{Y_n\}$ does not have accumulation points in $\hatT$.
\end{proof}

To discuss the properties of infinite horocoordinates,
 we briefly recall the work of Brock-Masur-Minsky \cite{BMM},
 which defines the ``end invariants" of Weil-Petersson geodesic rays.
 End invariants are defined through the theory of measured {\em laminations} rather than the foliations.
Recall that there are natural one-to-one correspondence between the space of measured foliations and the space of measured laminations(see e.g. \cite[Section 4]{CB} for details).
 The end invariants of WP geodesic rays are defined as an analogue of 
 the end invariants of quasi-Fuchsian manifolds that we discussed in \S \ref{sec.Bers-end}.
 In order to avoid the confusion, in this paper we call the end invariants in \cite{BMM} as {\em $\WPE$ invariants}.
Given a hyperbolic surface $X\in\T$, we call a family $P$ of disjoint simple closed curves {\em a Bers pants decomposition} for $X$ if
\begin{itemize}
	\item $X\setminus P$ is a union of pairs of pants, and
	\item $\ell_\alpha(X)<L$ for all $\alpha\in P$, where $L = L(S)$ is called the Bers constant which depends only on the topology of $S$.
\end{itemize} 
The existence of such pants decompositions is due to Bers 
(see e.g. \cite{Buser}).
Let $r:[0,\infty)\rightarrow\T$ be an infinite WP geodesic ray parametrized by arc length.
A simple closed curve $\alpha$ is called a {\em Bers curve} of $r$ 
if there exists $t\in[0,\omega)$ such that $\alpha$ is a curve in a Bers pants decomposition for $r(t)$.
Then the $\WPE$ invariant is the supporting laminations of 
the union of pinching curves and all the limits of Bers curves of $r$ in $\pmf$ as $t\rightarrow\infty$.
We refer \cite{BMM} for the precise definition and well-definedness of $\WPE$ invariants.

Let $D_+(X)$ denote the end point in $\hatT$ or the WP-end invariant of the ray $D(X)=:r:[0,\omega)\to\T$, where $\omega$ arrowed to be infinity.
We define $$\dwp(Z,D_+(X)) := \lim_{t\to\omega} \dwp(Z,r(t)).$$
Note that $\dwp(b,D_+(X))$ is finite if and only if $D_+(X)\in\hatT$.
\begin{thm}
	Let $(\omega,D)\in\pwp\T$.
	If the ray $D(X)$ is infinite for some $X\in\T$, then $(\omega,D)\in\pwp^\infty\T$.
	
	Conversely, if $(\omega,D)\in\pwp^\infty\T$, then we have
	$$\sup_{X\in\T} \dwp(b,D_+(X)) = \infty.$$
\end{thm}
\begin{proof}
	The first statement was proven in Theorem \ref{thm.wp-finite}.
	
	Suppose that $$\sup_{X\in\T} \dwp(b,D_+(X))<\infty.$$
	In this case, $D_+(X)\in\hatT$ for all $X$.
	Let $\{Y_n\}\subset\T$ be a sequence so that $\mathcal{WP}(Y_n)\rightarrow(\omega,D)\in\pwp^\infty\T$.
	For notational simplicity let $\omega_n := \omega_{Y_n}$ and $D_n:=D_{Y_n}$.
	Since $D_n(X)\rightarrow D(X)$, Lemma \ref{lem.wp-exp} implies that for any $\epsilon>0$ there exists $N>0$ such that for any $n\geq N$, we have the inverse of triangle inequality:
	\begin{equation}\label{eq.inv-tri}
			\dwp(X,Y_n)\geq \dwp(X,D_+(X))+\dwp(D_+(X),Y_n)-\epsilon.
	\end{equation}
	Hence we have
	\begin{align*}
		\omega_n(X) &= \dwp(X,Y_n) -\dwp(b,Y_n)\\
					&\geq \dwp(X,Y_n) - \dwp(b,D_+(X)) - \dwp(D_+(X),Y_n)~~\hspace{0.1cm}&\text{ (by triangle inequality)}\\
					&\geq \dwp(X,D_+(X)) -\epsilon - \dwp(b,D_+(X))&(\text{by }(5.1))\\
					&\geq -\dwp(b,D_+(X))-\epsilon
	\end{align*}
	Since $\epsilon$ was arbitrary and $\omega_n\rightarrow\omega$, we see that $\omega(X)\geq-\dwp(b,D_+(X))$.
	Hence $$\inf_{X\in\T}\omega(X)\geq-\sup_{X\in\T}\dwp(b,D_+(X))>-\infty$$ which implies that $(\omega,D)\in\pwp^\text{fin}\T$.
	
\end{proof}

For the case where WP-end invariants are uniquely ergodic laminations, the situation becomes simpler.
In \cite{BMM}, Brock-Masur-Minsky shows the following.
\begin{lem}[{\cite[Lemma 2.10 and Corollary 2.12]{BMM}}]\label{lem.wp-intersection}
	Let $X\in\T$ and 
	$r_n$ be WP-geodesic rays with $r_n(0) = X$ where $n\in\N\cup\{\infty\}$.
	Let $\mu_n$ be the $\WPE$ invariants of $r_n$.
	Suppose the initial tangent of $r_n$ converges to the initial tangent of $r_\infty$.
	Then any limit $\mu$ of the sequence $\{\mu_n\}$ in the $\pmf$ satisfies
	$$i(\mu,\mu_\infty) = 0.$$
\end{lem}

Contrary to the finite horofunctions, we have the following independence of the base points for uniquely ergodic case.
\begin{thm}\label{thm.wp-ue}
	Let $\{Y_n\}\subset\T$ be a sequence converging to $(\omega,D)\in\pwp\T$ in $\overline{\mathcal{WP}(\T})$.
	Suppose that there exists $X_0$ such that the ray $D(X_0)$ has a uniquely ergodic lamination $\lambda$ as its WP-end invariant.
	Then for any $X\in\T$, the WP-end invariant of $D(X)$ is $\lambda$.
\end{thm}
\begin{proof}
		Let $r:[0,\infty)\rightarrow\T$ denote the WP geodesic ray $D(X_0)$.
		By the work of Brock \cite{Brock-vol}, there is a universal bound $A>0$ such that 
		$\dwp(r(t), P_t)<A$ where $P_t\in\hatT$ is the maximal cusp corresponding to a Bers pants decomposition for $r(t)$.
		Let $X\in\T$ be an arbitrarily chosen point.
		Let $g_t$ (resp. $g'_t$) denote the WP-geodesic connecting $X$ and $P_t$ (resp. r(t)).
		By the CAT(0) property of $\hatT$, the initial tangent of $g_t$ and $g'_t$ converge to the same direction $D'$.
		The direction $D'$ should be $D(X)$ because CAT(0) property tells us that $\dwp(r(t),D(X)(t))\leq \dwp(X,X_0)$ (we apply \cite[Chapter II, Proposition 2.2]{BH} to the rays $D_{Y_n}(X_0)$ and $D_{Y_n}(X)$, and then take limit as $n\to\infty$).
		Then by the definition of WP-end invariants, we have $P_t\rightarrow\lambda$ since $\lambda$ is uniquely ergodic.
		By Lemma \ref{lem.wp-intersection}, we see that the WP-end invariant $\mu$ of $D(X)$ satisfies $i(\mu,\lambda) = 0$.
		Hence we see that $\mu = \lambda$ by Lemma \ref{lem.ue-inter}.
\end{proof}

Thanks to Theorem \ref{thm.wp-ue}, we may define uniquely ergodic points of $\pwp\T$.
\begin{defi}
	A horocoordinate $(\xi,D)\in \pwp\T$	 is {\em uniquely ergodic} if for some (hence any by Theorem \ref{thm.wp-ue}) $X$, the WP-end invariant of $D(X)$ is uniquely ergodic.
	For a uniquely ergodic horocoordinate corresponding to uniquely ergodic lamination $\lambda$, we say that $\lambda$ is {\em the WP-end invariant} of $(\omega,D)$.
\end{defi}

\begin{thm}\label{thm.wp-ue-thurston}
	Let $r:[0,\infty)\to\T$	be a WP geodesic ray which converges to 
	a uniquely ergodic lamination $\lambda$ in Thurston's compactification.
	Then $\{r(t)\}$ also converges to some $(\omega, D) \in \overline{\mathcal{WP}(\T)}$ which is uniquely ergodic with WP-end invariant $\lambda$.
\end{thm}
\begin{proof}
		First note that by \cite[Theorem 1.2]{BMM}, we see that for any $Y\in\T$, 
		the initial tangent of geodesics connecting $Y$ and $r(t)$ converges.
		Then by Corollary \ref{cor.int-formula-WP}, we see that $\{r(t)\}$ converges to some $(\omega,D)\in\overline{\mathcal{WP}(\T)}$.
		
		As $\{r(t)\}$ converges in Thurston's compactification, there exists $s_t\rightarrow 0$ such that $s_tr(t)$ converges to a measured lamination representative of $\lambda$.
		Let $P_t$ be a Bers pants decomposition of $r(t)$.
		By definition we have $\ell_{r(t)}(P_t) = i(r(t),P_t)\leq (3g-3)L$ where $L>0$ is the Bers constant ($(3g-3)$ is the number of curves in the pants decompositions of $S$).
		By taking a subsequence if necessary, we may suppose that $P_t$ also has a limit $\mu$ in $\mathcal{ML}(S)$.
		Then we have 
		$$0\leq i(\lambda,\mu) = \lim_{t\rightarrow\infty} i(s_tr(t),P_t)\leq \lim_{t\to\infty} s_t\cdot (3g-3)L = 0.$$
		By Lemma \ref{lem.ue-inter}, we see that $\lambda = \mu$ in $\pmf$.
		It holds for any accumulation point $\mu$ of $P_t$, and hence we see that $P_t$ converges to $\lambda$.
		Therefore we see that the WP-end invariant $D_+(r(0)) = \lambda$.
		Hence Theorem \ref{thm.wp-ue} implies that $(\omega,D)$ is uniquely ergodic with WP-end invariant $\lambda$.
\end{proof}

\begin{rmk}
	Let $r:[0,\infty)\rightarrow\T$ be a WP-geodesic ray with uniquely ergodic end invariant $\lambda$.
	It has pointed out by \cite{BLMR}
	that the limit set in the Thurston compactification of $r$ consists of the single point $\lambda$.
\end{rmk}

\section{Compactification of Teichm\"uller space via renormalized volume}\label{sec.main}
We have seen several boundaries of the Teichm\"uller space $\mathcal{T}(S)$.
We now apply the idea of horocoordinate boundary to the notion of the renormalized volume to construct a boundary of $\mathcal{T}(S)$.
To define a horocoordinate, it is necessary to have data of ``directions".
Recall that in Corollary \ref{cor.int-formula}, 
directions are the data that give derivatives.
Hence one may utilize the formula of the derivative of the renormalized volume (Theorem \ref{thm.dRvol})
in order to define ``directions".
Later in \S \ref{sec.wp-grad}, we will also consider WP gradient flows to see such formula of derivatives can actually be seen as directions.
\subsection{Lipschitz property of the renormalized volume}
In \cite{Schlenker} and \cite{Kojima-McShane}, Schlenker and Kojima-McShane proved that there are explicit upper bounds of the renormalized volume $\VR$ in terms of the WP metric and the Teichm\"uller distance.
Recall that $S$ is an orientable closed surface of genus $g\geq 2$.

\begin{thm}[{\cite[Theorem 5.4]{Schlenker}},\cite{Kojima-McShane}]\label{thm.Rvol-WP-T}
Let $X,Y\in\T$. Then we have
\begin{enumerate}
	\item $\VR(X,Y)\leq \Con \dwp(X,Y)$, and
	\item $\VR(X,Y)\leq \ConT d_{\mathcal{T}}(X,Y).$
\end{enumerate}
\end{thm}
To see the proof of Theorem \ref{thm.Rvol-WP-T}, let us summarise upper bounds of differentials given in Theorem \ref{thm.dRvol}.
Let $X\in\T$ and $q\in\QD(X)$.
For any $z\in X$, the point-wise norm is defined by
$$||q(z)||=\frac{|q(z)|}{\rho_X(z)},$$
where $\rho_X$ is the hyperbolic metric on $X$.
Then the Teichm\"uller metric is the $L^1$ metric and WP metric is the $L^2$ metric with respect to this point-wise metric.
Therefore we have
\begin{lem}[{\cite[Proof of Theorem 1.2]{Schlenker-MRL}}, {\cite[Proof of Theorem 1.4]{Kojima-McShane}}]\label{lem.diff-bound}
	Let $Y\in\T$ and $\sigma:[0,T]\rightarrow\T$ be a differentiable path.
	\begin{enumerate}
		\item If $\sigma$ is a geodesic with respect to the Teichm\"uller metric, then $$\left|\frac{d}{dt}\VR(\sigma(t),Y)\right|\leq \ConT.$$
		\item If $\sigma$ is a geodesic with respect to the WP metric, then $$\left|\frac{d}{dt}\VR(\sigma(t),Y)\right|\leq \Con.$$
	\end{enumerate}
	Here, we suppose that geodesics are parametrized by their arc length.
\end{lem}
\begin{proof}
	This is essentially Nehari's inequality, Theorem \ref{thm.Nehari} and Lemma \ref{lem.Nehari}.
\end{proof}
Then Theorem \ref{thm.Rvol-WP-T} is obtained by integrating quantities in Lemma \ref{lem.diff-bound} along corresponding geodesic segments.

\subsection{Embed Teichm\"uller space into space of Lipschitz functions}\label{sec.cpt}
Imitating horofunctions defined with distances, we define a function on $\T$ via the renormalized volume as follows.
Let us fix a base point $b\in\T$.
\begin{defi}
Let $Z\in\T$. We define $\nu_{Z}:\mathcal{T}(S)\rightarrow\mathbb{R}$ by
$$\nu_{Z}(X):=\VR(X,Z)-\VR(b,Z)$$
for $X\in\T$.
We call $\nu_{Z}$ a {\em volume horofunction}.
\end{defi}

The variation formula of $\VR$ (Theorem \ref{thm.dRvol}) gives the following integral expression of $\nu_{Z}$:
\begin{prop}\label{prop.int1}
Let $X,Z\in\T$ and let $\sigma:[0,T]\rightarrow\T$ be a piecewise differentiable path connecting $X$ and $b$.
Then 
$$\nu_{Z}(X):=\int_0^T-\mathrm{Re}\langle q_Z(\sigma(t)),\dot\sigma(t)\rangle dt.$$
\end{prop}

\begin{proof}
Let $\sigma_{1}:[0,T_1]\rightarrow\T$ be a piecewise differentiable path connecting $b$ and $Z$.
As the renormalized volume function $\VR$ is smooth, the derivation formula (Theorem \ref{thm.dRvol}) implies
\begin{align*}
\VR(X,Z) &= \int_0^T-\mathrm{Re}\langle q_Z(\sigma(t)), \dot\sigma(t) \rangle dt  +
\int_0^{T_1}-\mathrm{Re}\langle q_Z(\sigma_{1}(t)), \dot\sigma_{1}(t) \rangle dt\\
&= \int_0^T-\mathrm{Re}\langle q_Z(\sigma(t)), \dot\sigma(t) \rangle dt + \VR(b,Z).
\end{align*}
\end{proof}

By Proposition \ref{prop.int1}, we see that the function $\nu_Z$ is a Lipschitz map:
\begin{prop}\label{prop.Lip}
	The function $\nu_Z:\T\rightarrow\R$ is a Lipschitz map with respect to both the Teichm\"uller metric and the WP metric, i.e.
	\begin{enumerate}
		\item $|\nu_Z(X)-\nu_Z(Y)|\leq\Con \dwp(X,Y)$, and
		\item $|\nu_Z(X)-\nu_Z(Y)|\leq\ConT \dt(X,Y)$.
	\end{enumerate}
\end{prop}
\begin{proof}
	By Lemma \ref{lem.diff-bound} and Proposition \ref{prop.int1},
	 one obtains these bound by integrating the differential of $\VR$ (Theorem \ref{thm.dRvol}) along the Teichm\"uller geodesic or the WP geodesic connecting $X$ and $Y$.
\end{proof}

Thus we see that $\nu_{Z}$ is a Lipchitz function and vanishes at the base point $b$.
From now on we consider the WP metric on $\T$ and Lipchitz functions with respect to the WP metric.
Let $\Lip\T$ denote the space of $C$-Lipchitz functions on $\T$ for $C = \Con$ which vanishes at $b$.
We have a map
$$\mathcal{V'}:\T\rightarrow\Lip\T$$
defined by $\mathcal{V'}(Z):=\nu_{Z}$.

\begin{prop}\label{prop.horo}
The map $\mathcal{V'}:\T\rightarrow\Lip\T$ is injective and continuous.
\end{prop}
\begin{proof}
For any $X,Y,Z\in\T$, we have by Proposition \ref{prop.Lip}
\begin{align*}
|\nu_{X}(Z)-\nu_{Y}(Z)|&\leq|\VR(Z,X) - \VR(Z,Y)| + |\VR(X,b)-\VR(Y,b)|\\
&\leq 2\cdot\Con \dwp(X,Y).
\end{align*}
So if $Y_{n}\rightarrow Y$ then $\nu_{Y_{n}}\rightarrow\nu_{Y}$ uniformly on $\T$.
Hence $\mathcal{V'}$ is continuous.

As pointed out in \cite{KS,BBB}, it can be readily seen from the differential formula (Theorem \ref{thm.dRvol}) that the function
$\VR(\cdot, Z):\T\rightarrow\R$ has its critical point only at $Z$.
This implies that $Z$ is the unique critical point of the function $\nu_Z$.
Therefore $\mathcal{V'}$ is injective.
\end{proof}

\begin{rmk}\label{rmk.positive}
	Recently in \cite{BBB,BBP}, it is shown
	 that the renormalized volume is a non-negative function on $\TT$ and it is zero only on the diagonal i.e. $\VR(X,Y)\geq 0$ and $\VR(X,Y) = 0$ if and only if $X = Y$.
	Then the injectivity of $\mathcal{V}'$ also follows by looking at the minimum of $\nu_Z$, similarly to the standard arguments for the case of horofunctions of distances.
	\end{rmk}

Recall that for $C = \Con$, 
$$\Lip\T\subset\prod_{x\in\T}[-C\cdot\dwp(b,X),C\cdot\dwp(b,X)],$$ and  $q_{X}(Y)\in\QD_{B}(Y)$\text{ (Bers embedding).}
Then we now define a map from $\T$ to the space $\LQ$ in Definition \ref{defi.LQ}.
\begin{defi}\label{def.vhoro}
We define a map
$$\mathcal{V}:\T\rightarrow\LQ$$
by 
$\mathcal{V}(Z) = (\nu_{Z}(X), q_{Z}(X))_{X\in\T}$.
\end{defi}

Then the map $\mathcal{V}$ is an embedding:
\begin{prop}\label{prop.Nu}
The map $\mathcal{V}:\T\rightarrow\LQ$ 
 is a homeomorphism onto its image.
\end{prop}
\begin{proof}
Injectivity, and continuity to the first coordinate $\prod_{X\in\T}[-C\cdot\dwp(b,X),C\cdot\dwp(b,X)]$  of $\LQ$
follow from Proposition \ref{prop.horo}.
Recall that $q_{(\cdot)}(X):\T\rightarrow\QD_{B}(X)$ defines the Bers embedding at $X$.
As $\LQ$ is equipped with the product topology, we see that the map $\mathcal{V}$ is continuous.
Suppose that $(\nu_{Z_{n}}, q_{Z_{n}}(X))$ converges to $(\nu_{Z},q_{Z}(X))$.
As $q_{Z_{n}}(X)\rightarrow q_{Z}(X)$ implies $Z_{n}\rightarrow Z$, we see that $\mathcal{V}^{-1}(\nu_{Z_n},q_{Z_n}) = Z_n\rightarrow Z = \mathcal{V}^{-1}(\nu_{Z},q_{Z})$.
Thus the inverse $\mathcal{V}^{-1}$ is continuous on $\mathcal{V}(\T)$.
\end{proof}

By Proposition \ref{prop.Nu} and \ref{prop.LQ}, the closure $\overline{\mathcal{V}(\T)}$ is compact.
Thus we get the desired compactification of $\T$ and Theorem \ref{thm.main-cpt} is proved.
\begin{defi}
We denote the closure by $\overline{\T}^{\mathrm{vh}}:=\overline{\mathcal{V}(\T)}$ (\underline volume and \underline horo) and 
the boundary by $\partial_{\mathrm{vh}}\T:=\overline{\mathcal{V}(\T)}\setminus \mathcal{V}(\T)$.
\end{defi}

The construction of $\Tvh$ is compatible with the action of the mapping class group $\mcg$ on $\T$.
\begin{prop}\label{prop.bh-mcg-action}
The action of $\mcg$ on $\T$ extends to a continuous action by homeomorphisms on $\Tvh$ by
\begin{align}\label{eq.action}
\psi\cdot\nu(X)&:=\nu(\psi^{-1}X) - \nu(\psi^{-1}b)\text{ for each $X\in\T$}\\\label{eq.action2}
 \psi\cdot q(X) &:= \psi_{*}q(\psi^{-1}X),
\end{align}
and $\psi(\nu,q) = (\psi\cdot\nu, \psi\cdot q)$,
where $\psi_*$ is the push-forward.
\end{prop}
\begin{proof}
Although the proof of the continuity for the first coordinate ($\nu$ of $(\nu,q)$) goes similarly to Lemma \ref{lem.action},
we demonstrate here for completeness.
Let $\psi\in\mcg$, and $\nu_{Z}$ be as defined in Definition \ref{def.vhoro}.
We define $$\psi\cdot\nu_{Z}:=\nu_{\psi Z}.$$
Then 
\begin{align*}
\psi\cdot\nu_{Z} (X) = \nu_{\psi Z}(X)
&= \VR(X,\psi Z)-\VR(b,\psi Z)\\
&= \VR(\psi^{-1}X, Z) - \VR(b,Z) + \VR(b,Z) - \VR(\psi^{-1}b, Z)\\
&= \nu_{Z}(\psi^{-1}X)-\nu_{Z}(\psi^{-1}b).
\end{align*}
For $Z\in\T$, the action of $\psi$ on $q_Z$ should be
$$\psi\cdot q_Z(X) = q_{\psi Z}(X) = \psi_*(q_Z(\psi^{-1}X)).$$
Hence the continuity of the action of $\psi$ on $\T$ implies that the actions defined in (\ref{eq.action}) and (\ref{eq.action2}) are continuous.
In other words, we have that if $(\nu_{n},q_{n})\rightarrow (\nu, q)$ in $\LQ$, then $\psi\cdot (\nu_{n},q_{n})\rightarrow \psi\cdot(\nu, q)$ in $\LQ$.
\end{proof}

The integral formula in Proposition \ref{prop.int1} extends to the boundary $\pvh\T$.
\begin{thm}\label{thm.integral}
Let $(\nu,q)\in\pvh\T$ and $\sigma:[0,T]\rightarrow \T$ be a piecewise differentiable path connecting $b$ and $X$.
Then we have
$$\nu(X)=\int_0^T-\mathrm{Re}\langle q(\sigma(t)),\dot\sigma(t)\rangle dt.$$
\end{thm}
\begin{proof}
Let $(\nu_{n},q_{n})\rightarrow(\nu,q)$ be a sequence with $(\nu_{n},q_{n})\in\mathcal{V}(\T)$.
By Proposition \ref{prop.int1}, for each $n$ we have
$$\nu_{n}(X)=\int_0^T-\mathrm{Re}\langle q_{n}(\sigma(t)),\dot\sigma(t)\rangle dt.$$
Hence by the dominated convergence theorem, we have the conclusion.
\end{proof}

Finally we remark the following.
\begin{coro}
The identity map on $\T$ does not extend to a homeomorphism from $\overline{\mathcal{WP}(\T)}$ to $\Tvh$.
\end{coro}
\begin{proof}
	Brock \cite[Theorem 1.8]{Brock2}  has shown that the identity map does not extend to a homeomorphism from a WP visual sphere to a Bers boundary. Hence the same holds for $\overline{\mathcal{WP}(\T)}$ and $\Tvh$.
\end{proof}
\subsection{(In)finite points in the boundary and Weil-Petersson gradient flows}\label{sec.wp-grad}
Similarly to the case of Weil-Petersson metric, we have the following decomposition:

\begin{align*}
	\pvh^\mathrm{fin}\T &:= \{(\nu, q)\in\pvh\mathcal T(S)\mid \inf_{X\in\T}\nu(X)>-\infty\} \text{ and, }\\
	\pvh^\infty\T &:= \pvh\T\setminus\pvh^\mathrm{fin}\T.
\end{align*}

In \S\ref{sec.WP}, we have discussed (in)finite horocoordinates by looking at the WP geodesic rays determined by the directions.
When we study $\pvh\T$, one natural object corresponding to those geodesic rays is the notion of {\em Weil-Petersson(WP) gradient flows} of $-\VR$.
As we have seen in \S\ref{sec. QF and S}, 
the space of quadratic differentials $\QD(S)$ can be identified with the cotangent bundle of the Teichm\"uller space $\T$.
The $q$ of each $(\nu,q)\in\Tvh$ determines a section of the cotangent bundle $\QD(S)$, and hence gives a Weil-Petersson gradient flow see e.g. \cite{BBB,BBP}.
It was proved very recently by Bridgeman-Bromberg-Pallete \cite{BBP} that for $(\nu_X,q_X)\in\Tvh$ corresponding to $X\in\T$, the unique attracting point of the WP gradient flow determined by $q_X$ is $X$.
This work shows that it is natural to regard $q$
 of $(\nu,q)\in\Tvh$ as a 
 ``direction" associated to $\nu$.
As it is remarked in \cite{BBB} the WP gradient flow line exists all the time due to the completeness of the Teichm\"uller metric.
Along the flow $X_t$, the differential of $-\VR$ is equal to $||q(X_t)||_2^2$, see \cite{BBB,BBP}.

 \begin{defi}
 	Let $(\nu,q)\in\Tvh$ and $X\in\T$. 
 	Let $X_t$ denote the WP gradient flow line given by $q$ starting at $X$ i.e. $X_0=X$.
 	Then we say that $q$ gives a {\em finite flow} at $X$ if
 	$$\int_0^\infty ||q(X_t)||^2_2dt<\infty.$$
 	Conversely, we say that $q$ determines an {\em infinite flow} at $X$ if 
 	$$\int_0^\infty ||q(X_t)||^2_2dt=\infty.$$
 \end{defi}

With this definition, we have the following as an easy consequence of the works in \cite{BBB,BBP}.

\begin{prop}[c.f. \cite{BBB,BBP}]
	Let $X\in\T$ and $(\nu_X,q_X)\in\Tvh$ associated horocoordinate.
	Then for every $Y\in\T$, $q_X$ determines a finite flow at $Y$.
\end{prop}
\begin{proof}
	For every $Y$, the flow line determined by $q_X(Y)$ terminates at $X$ by \cite{BBP}.
	In this case we have $\int_0^\infty ||q(X_t)||^2_2dt=-\VR(X,Y)$,
	and hence it is finite.
\end{proof}

\begin{thm}\label{thm.vh-finite}

\begin{enumerate}
	\item For each $Y_\infty\in\hatT\setminus\T$, 
	we may define $\VR(\cdot, Y_\infty):\T\rightarrow\R$ and 
	$\nu_\infty(X):=\VR(X,Y_\infty)-\VR(b,Y_\infty)$ is in $\pvh^\mathrm{fin}\T$.
	\item Let $(\nu,q)\in\Tvh$. Suppose there exists $X\in\T$ such that $q$ gives an infinite flow line at $X$. Then we have $(\nu,q)\in\pvh^\infty\T$.
\end{enumerate}
\end{thm}
\begin{proof}
Let $\{Y_n\}$ be a Cauchy sequence with respect to the WP metric with limit $Y_\infty\in\hatT$.
As we proved in the proof of  Proposition \ref{prop.horo}, we have for any $X\in\T$
$$|\nu_{Y_n}(X) - \nu_{Y_m}(X)|\leq 2\cdot\Con\cdot\dwp(Y_n, Y_m).$$
Therefore $\nu_{Y_n}$ converges to some volume horofunction $\nu_\infty$ which depends only on the limit $Y_\infty\in\hatT$ of $\{Y_n\}$.
Since $\mathcal{V}$ is a homeomorphism onto its image, $\nu_\infty\in\pvh\T$.
As $$|\VR(X, Y_n) - \VR(X, Y_m)|\leq \Con\cdot\dwp(Y_n, Y_m)$$
for every $X\in\T$, we may define $\VR(X,Y_\infty)$.
Then we have $\nu_\infty(X) = \VR(X,Y_\infty)-\VR(b,Y_\infty)$, which  implies $\nu_\infty\in\pvh^\mathrm{fin}\T$ since $\VR(b,Y_\infty)$ is finite.


For (2), by definition we have $\int_0^\infty||q(X_t)||_2^2 dt = \infty$ 
along the WP geodesic flow line $X_t$ starting at $X$ determined by $q$.
Then by Theorem \ref{thm.integral}, considering a path first connecting $b$ and $X=X_0$ and then following the WP flow line $X_t$, we have
$$\nu(X_T) = \nu(X_0)-\int^T_0||q(X_t)||^2_2dt.$$
Hence we have $\nu(X_T)\rightarrow -\infty$ as $T\rightarrow\infty$,
which implies $(\nu,q)\in\pvh^\infty\T$.

\end{proof}

\begin{rmk}
By the work of Kerckhoff-Thurston, in general, the limit point of a sequence $\{Y_{n}\}\subset\T$ in $\partial_{B}^{X}\T$ may depend on $X$.
This phenomena happen when for example $\{Y_{n}\}$ is obtained by iterations of Dehn twists.
It is interesting to characterize $\nu$ in $(\nu,q)\in\pvh^\mathrm{fin}\T$ for general case precisely.
\end{rmk}

\subsection{Volume of mapping tori and volume horofunctions}
A mapping class $\psi\in\mcg$ is called {\em pseudo-Anosov}
if $\psi$ has exactly two fixed points $F_{+}(\psi), F_{-}(\psi)\in\pmf$ which we may characterize as $\lim_{n\rightarrow\infty}\psi^n(X) = F_+(\psi)$ and
$\lim_{n\rightarrow-\infty}\psi^n(X) = F_-(\psi)$
 for any $X\in\T$ in the Thurston compactification.
Thurston has shown that the mapping torus 
$$M(\psi):= S\times[0,1]/((\psi(x),0)\sim (x,1))$$
admits a complete hyperbolic metric of finite volume.
Let $\Vol(M(\psi))$ denote the hyperbolic volume of $M(\psi)$.
First, we recall the following.
\begin{prop}\label{prop.KoMA}
Let $\psi\in\mcg$ be pseudo-Anosov.
Then $$\lim_{n\rightarrow\infty}\frac{1}{n}\VR(b,\psi^{-n}b) = \Vol(M(\psi)).$$
\end{prop}
\begin{proof}
In \cite{BB} and {\cite[Appendix]{Kojima-McShane}}, the version of convex core volume is proven.
Then in \cite{BC} or \cite{Schlenker-MRL}, it is proved that the renormalized volume and the convex core volume differ at most finite amount.
\end{proof}

Let us recall the work of Ohshika \cite{Ohshika} on so-called the reduced Bers boundary.
As we have recalled in \S \ref{sec.Bers-end}, associated to each point in the Bers boundary $\partial_B^X\T$, we have the end invariant which is a union of geodesic laminations, parabolic loci, and conformal structures on subsurfaces of $S$.
By collapsing each deformation spaces of conformal structures on subsurfaces, Ohshika defined the {\em reduced Bers boundary} denoted by $\partial_{RB}\T$.
Let us denote $$O_X:\partial_B^X\T\rightarrow \partial_{RB}\T$$ the projection map. The space $\partial_{RB}\T$ is equipped with the quotient topology by $O_X$.
As the notation suggests, Ohshika showed that the reduced Bers boundary $\partial_{RB}\T$ is independent of the base point $X$.
As all the information remained are geodesic laminations and parabolic loci (which are multicurves), we may regard $\partial_{RB}\T\subset \mathcal{GL}(S)$, where $\mathcal{GL}(S)$ is the space of geodesic laminations on $S$.
Note that Ohshika showed that the topology of $\partial_{RB}\T$ is not compatible with the inclusion $\partial_{RB}\T\subset\mathcal{GL}(S)$.
We refer the original paper by Ohshika \cite{Ohshika} for the precise definition and properties.
We have the following natural projection from $\pvh\T$ to $\partial_{RB}\T$:
\begin{lem}\label{lem.RB-ue}
	There is a natural continuous $\mcg$-equivariant map 
	$$\mathcal{O}:\pvh\T\rightarrow\partial_{RB}\T$$	
	defined by $\mathcal{O}(\nu,q) = O_X(q(X))$ for some (in fact, any) $X$.
	Moreover, $\mathcal O$ is bijective on $\mathcal{O}^{-1}(\mathcal{UE}(S))$ where $\mathcal{UE}(S)$ is the space of uniquely ergodic geodesic laminations.
	Furthermore, if $\psi\in\mcg$ is pseudo-Anosov, then provided $\mathcal{O}(\nu,q)\not=F_-(\psi)$, 
	$\psi^n(\nu,q)$ converges to $\mathcal{O}^{-1}(F_+(\psi))$ as $n\rightarrow\infty$.
\end{lem}
Assuming Lemma \ref{lem.RB-ue}, let us define uniquely ergodic points.
\begin{defi}	
	A point $(\nu,q)\in\partial_\mathrm{vh}\T$	is said to be {\em uniquely ergodic} if the image $\mathcal{O}(\nu,q)$ is a uniquely ergodic geodesic lamination.
\end{defi}

\begin{proof}[Proof of Lemma \ref{lem.RB-ue}]
	Let $(\nu,q)\in\pvh\T$. 
	Note that for each $X\in\T$, $q(X)\in\partial_B^X\T$.
	One key ingredient of the construction of $\partial_{RB}\T$ by Ohshika is that the image $O_X(q(X))$ is independent of $X$ \cite[Theorem 3.7]{Ohshika}.
	Hence $\mathcal{O}$ is well-defined and independent of the choice of $X$.
	Since $\partial_{RB}\T$ has the quotient topology by $O_X$, $\mathcal{O}$ is continuous by definition.
	Ohshika also showed that the action of $\mcg$ on $\T$ extends continuously to $\partial_{RB}\T$ \cite[Corollary 3.8]{Ohshika}.
	Since the same holds for $\mcg$ action to $\pvh\T$, 
	$\mathcal{O}$ is $\mcg$-equivariant.
	
	Since the uniquely ergodic points in $\pvh\T$ are uniquely determined by their supporting geodesic lamination, Theorem \ref{thm.elt} implies that $\mathcal{O}$ is bijective on the space of uniquely ergodic points.
	
	Now, let $(\nu,q)\in\Tvh$ and suppose that $\mathcal O(\nu,q)\not=F_-(\psi)$.
	Consider a sequence $\{Y_n\}\subset\T$ with $Y_n\rightarrow(\nu,q)$ in $\Tvh$.
	Since $\Tvh$ is compact and metrizable (Proposition \ref{prop.LQ}), $\{\psi^n(\nu,q)\}$ converges as $n\rightarrow\infty$ to some $(\nu_\infty,q_\infty)\in\Tvh$ after taking subsequence if necessary.
	Then by the diagonal argument, we may find a sequence $\psi^N Y_{k_N}$ which converges to $(\nu_\infty,q_\infty)$.
	As $F_+(\psi)$ is the unique attracting point of $\psi$ in the Thurston compactification, we may suppose that $\psi^NY_{k_N}$ also converges to $F_+(\psi)$ in $\T\cup\partial_{\rm Th}\T$ after taking a subsequence.
	Then by the work of Brock (Theorem \ref{thm.Brock}) and Ohshika \cite{Ohshika}, we have that $\mathcal O(\nu_\infty,q_\infty) = F_+(\psi)$ since $F_+(\psi)$ is uniquely ergodic (see also Lemma \ref{lem.ue-inter}).
	Hence we have $(\nu_\infty,q_\infty) = \mathcal O^{-1}F_+(\psi)$ for any limit point $(\nu_\infty,q_\infty)$ of $\psi^n(\nu,q)$.
	Therefore $\psi^n(\nu,q)$ itself converges to $\mathcal O^{-1}(F_+(\psi))$.
\end{proof}

\begin{thm}\label{thm.lim-nu}
Let $\psi\in\mcg$ be pseudo-Anosov and $(\nu,q)\not=\mathcal{O}^{-1}(F_-(\psi))\in\Tvh$.
Let $(\nu_+,q_+):=\mathcal{O}^{-1}(F_+(\psi))$.
Then $$\lim_{n\rightarrow\infty}\frac{1}{n}\nu(\psi^{-n}b) = \nu_{+}(\psi^{-1}b) = \Vol(M(\psi)).$$
\end{thm}
\begin{proof}
Suppose first that $\nu=\nu_Z$ for some $Z\in\T$.
By definition, $\nu(\psi^{-n}b) = \VR(\psi^{-n}b,Z)-\VR(b,Z)$.
As $\left|\VR(\psi^{-n}b,Z) - \VR(\psi^{-n}b,b)\right|\leq \Con\dwp(Z,b)$ by Proposition \ref{prop.Lip}, 
we have 
\begin{equation}\label{eq.limit-nu}
\lim_{n\rightarrow\infty}\frac{1}{n}\nu(\psi^{-n}b) = \Vol(M(\psi)),
\end{equation}
by Proposition \ref{prop.KoMA}.
Now suppose $(\nu,q)\in\Tvh$.
Note that by $\psi^i\nu(\psi^{-1}b) = \nu(\psi^{-i-1}b)  - \nu(\psi^{-i}b)$, we have
$$\nu(\psi^{-n}b) = \sum_{i=0}^{n-1}\psi^i\cdot\nu(\psi^{-1}b).$$

As $\psi^i(\nu,q)\rightarrow \nu_+$ (Lemma \ref{lem.RB-ue}), for any $\epsilon>0$ there exists $N\in\mathbb{N}$ such that for any $m\geq N$, 
$\left|\psi^m\nu(\psi^{-1}b)-\nu_+(\psi^{-1}b)\right|<\epsilon$.
Hence 
\begin{align*}
&\lim_{n\rightarrow\infty}\frac{1}{n}\left|\nu(\psi^{-n}b) - n\nu_+(\psi^{-1}b)\right|\\
= &\lim_{n\rightarrow\infty}\frac{1}{n}\left|\sum_{i=0}^{N}\left(\psi^i\cdot\nu(\psi^{-1}b) - \nu_+(\psi^{-1}b)\right) + 
\sum_{i=N}^{n-1}\left(\psi^i\cdot\nu(\psi^{-1}b)-\nu_+(\psi^{-1}b)\right)\right|
 \leq \epsilon.
\end{align*}

 This is true for any $\epsilon>0$, and hence 
 $$\lim_{n\rightarrow\infty}\frac{1}{n}\nu(\psi^{-n}b) = \nu_+(\psi^{-1}b).$$
 This holds regardless $(\nu,q)\in\Tvh\setminus\pvh\T$ or $(\nu,q)\in\pvh\T$, and thus by (\ref{eq.limit-nu}) we have
$$\lim_{n\rightarrow\infty}\frac{1}{n}\nu(\psi^{-n}b) = \nu_+(\psi^{-1}b) = \Vol(M(\psi)).$$
\end{proof}

\begin{coro}\label{coro.lim_nu}
Let $\psi\in\mcg$ be pseudo-Anosov and let $(\nu_-,q_-):=\mathcal{O}^{-1}(F_-(\psi))$, then we have
$$\nu_-(\psi^{-1}b) = -\Vol(M(\psi)).$$
\end{coro}
\begin{proof}
	Notice that $\psi(\nu_-,q_-) = (\nu_-,q_-)$.
	Hence $$\nu_-(\psi^{-1}b) = \psi^{-1}\cdot\nu_-(\psi^{-1}b) 
	= \nu_-(b) - \nu_-(\psi (b)) = -\nu_-(\psi(b)).$$
	Then as $F_+(\psi^{-1}) = F_-(\psi)$, we have $\nu_-(\psi(b)) = \Vol(M(\psi))$ by Theorem \ref{thm.lim-nu}.
\end{proof}

\begin{rmk}
Theorem \ref{thm.lim-nu} and Corollary \ref{coro.lim_nu} holds independent of the choice of the base point $b$.
This kind of phenomenon is also observed for the standard horofunction of distances and translation lengths.
\end{rmk}

\section{A distance on $\T$ via renormalized volume}\label{sec.dist}
\subsection{Renormalized volume function $V_R$ is not a distance}\label{sec.not-distance}
As discussed above, the renormalized volume of quasi-Fuchsian manifolds defines a function
$$\VR:\T\times\T\rightarrow \mathbb{R}.$$
The function $\VR$ satisfies
\begin{itemize}
\item $\VR(X,Y)\geq 0$ and $\VR(X,Y) = 0$ if and only if $X = Y$ (\cite{BBB,BBP}), and
\item $\VR(X,Y) = \VR(Y,X)$ (by definition of quasi-Fuchsian manifolds).
\end{itemize}
Therefore it is natural to ask if $\VR$ defines a distance on $\T$ 
(see e.g. \cite[Problem 5.7(Agol)]{DHM}).

We now prove that $\VR$ does not satisfy the triangle inequality.
\begin{lem}\label{lem.not-distance}
Given any two points $X,X'\in\T$ and any $\epsilon>0$.
There exists a sequence of points $\{Y_{i}\}_{i=0}^{n}\subset \T$ with $Y_{0} = X$ and $Y_{n} = X'$ so that
$$\sum_{i=0}^{n-1}\VR(Y_{i},Y_{i+1})\leq \epsilon. $$
\end{lem}
\begin{proof}
	Let $\sigma:[0,T]\rightarrow\T$ denote the Weil-Petersson geodesic with $\sigma(0) = X$ and $\sigma(T) = X'$ which is parametrized by arc length.
	Note that the function $\VR(Z,\cdot):\T\rightarrow\T$ is smooth with critical value at $Z$.
	Then for each $t\in[0,T]$, there exists $\delta_t>0$ such that 
	$\frac{d}{ds}\VR(\sigma(t),\sigma(s))<\epsilon/T$ for any $s$ with $|t-s|<\delta_t$.
	As $\sigma([0,T])$ is compact, we may take finite points 
	$0 = t_0<t_1<\cdots<t_{n-1}<t_n=T$ so that 
	$$\VR(\sigma(t_i),\sigma(t_{i+1}))\leq \epsilon/T\cdot|t_{i+1}-t_i|$$
	for all $0\leq i<n$, which gives 
	$$\sum_i^{n-1}\VR(\sigma(t_i),\sigma(t_{i+1}))\leq\epsilon.$$
\end{proof}
Lemma \ref{lem.not-distance} says that if we divide the WP geodesic in a very small pieces, 
the total of the renormalized volume defined on each piece can be arbitrarily small.
Hence as an immediate consequence of Lemma \ref{lem.not-distance}, we have
\begin{thm}
The function $\VR:\T\times\T\rightarrow\mathbb{R}$ does NOT satisfy the triangle inequality.
\end{thm}

\subsection{A distance via the renormalized volume}
We now define a distance on $\T$.
\begin{defi}\label{def.dv}
	Given $X,Y\in\T$	, let
	$$d_R(X,Y):=\sup_{(\nu,q)} \nu(X)-\nu(Y),$$
	where the supremum is taken over $(\nu,q)\in\Tvh$.
\end{defi}
\begin{rmk}
We remark that as $\Tvh$ is compact the supremum is actually attained by some $(\nu,q)\in\Tvh$. Hence for any piecewise differentiable path $\sigma:[0,T]\rightarrow\T$ connecting $X$ and $Y$, we have
	\begin{equation}\label{eq.dv-int}
		d_R(X,Y) = \int_0^T -\mathrm{Re}\langle q(\sigma(t)),\dot\sigma(t)\rangle dt,
	\end{equation}
	for some $(\nu,q)\in\Tvh$.
Note also that if one takes the supremum over $\T$ (not $\Tvh$), one still gets the same distance as $\T\subset\Tvh$ is open dense.

It is also worth mentioning that if one considers the horofunctions with respect to a distance, say $d$, then the distance defined similarly to the one in Definition \ref{def.dv} recovers the original distance $d$ by the triangle inequality.
Due to the lack of the triangle inequality for $\VR$, the function $d_R$ differs from $\VR$.
\end{rmk}

\begin{thm}\label{thm.d_R-bound}
	We have the following estimates of $d_R$ in terms of $\dwp$, $\dt$, and $\VR$.
	\begin{enumerate}
		\item $d_R(X,Y)\leq\Con\dwp(X,Y)$,
		\item $d_R(X,Y)\leq\ConT\dt(X,Y)$,
		\item $\VR(X,Y)\leq d_R(X,Y)$.
	\end{enumerate}
\end{thm}
\begin{proof}
	By integral formula (\ref{eq.dv-int}) and Lemma \ref{lem.diff-bound}, we have upper bounds (1) and (2).
	Also by definition, we have $\nu_Y(X) - \nu_Y(Y) = \VR(X,Y)$, and hence (3) follows from the definition of $d_R$.
\end{proof}

\begin{thm}\label{thm.dv-dist}
	The function $d_R:\T\times\T\rightarrow\R$ gives a (possibly asymmetric) distance, that is: for any $X,Y,Z\in\T$, we have
	\begin{enumerate}
		\item $d_R(X,Y)\geq 0$ and $d_R(X,Y)=0\iff X = Y$.
		\item $d_R(X,Y)\leq d_R(X,Z) + d_R(Z,Y)$.
	\end{enumerate}
	Furthermore, the action of the mapping class group $\mcg$ on $(\T,d_R)$ is by isometries.
\end{thm}
\begin{proof}
	As we have noted above, results \cite{BBB,BBP} mentioned in Remark \ref{rmk.positive} imply that $\VR(X,Y)\geq 0$ and
	$\VR(X,Y) = 0 \iff X = Y$.
	By Theorem \ref{thm.d_R-bound} we have $d_R(X,Y)\geq \VR(X,Y)\geq 0$ and 
	hence we have
	$$d_R(X,Y)=0\Rightarrow \VR(X,Y) =0\Rightarrow X = Y.$$
	As $d_R(X,X) = 0$ by definition, we have property (1).
	
	Now let us consider triangle inequality (2).
	The quantity $d_R(X,Y)$ is expressed as $\nu(X)-\nu(Y)$ for some $(\nu,q)\in\Tvh$.
	Then 
	\begin{align*}
		d_R(X,Y) &= \nu(X) - \nu(Y)\\
				 &= (\nu(X) - \nu(Z)) + (\nu(Z)-\nu(Y))\\
				 &\leq  \sup_{(\nu',q')}(\nu'(X) - \nu'(Z)) + \sup_{(\nu',q')}(\nu'(Z)-\nu'(Y))\\
				 &= d_R(X,Z) + d_R(Z,Y).
	\end{align*}
	
	Given $\psi\in \mcg$ and $(\nu,q)\in\Tvh$, we have 
	$$\nu(\psi X)-\nu(\psi Y) = \psi^{-1}\nu(X)-\psi^{-1}\nu(Y).$$
	As $d_R(X,Y)$ is the supremum of $\nu(X)-\nu(Y)$ over $(\nu,q)\in\Tvh$, and $\psi(\Tvh) = \Tvh$, we see that $d_R(\psi X,\psi Y) = d_R(X,Y)$ for any $X,Y\in\T$ and $\psi\in\mcg$.
\end{proof}
\begin{rmk}
	The definition of $d_R$	may be compared with the expression as ratios of the Teichm\"ulelr distance (\ref{eq.Kerckhoff})  and Thurston's distance (\ref{eq.dTh}), and characterization of horofunctions in \cite{Wal, LS}.
\end{rmk}

\begin{rmk}
	Although we focus on $d_R$ in this paper, $d_R$ could be asymmetric. One may consider the symmetrization of $d_R$ by 
	$$\bar d_R(X,Y) = \sup_{(\nu,q)\in\Tvh}\left|\nu(X)-\nu(Y)\right|.$$
	
\end{rmk}

\begin{thm}\label{thm.wp-dv}
	The distance $d_R$  is quasi isometric to the WP distance $\dwp$.
	More precisely, there exists constants $L\geq 1$ and $K\geq 0$ which depends only on $S$ such that
	$$\frac{1}{L}\dwp(X,Y)-K\leq d_R(X,Y)\leq \Con\dwp(X,Y).$$
\end{thm}
\begin{proof}
	For any given $X,Y\in\T$, let $V_C(X,Y)$ denote the volume of the convex core of the quasi-Fuchsian manifold $q(X,Y)$.
	Then by the work of Brock \cite{Brock-vol}, we see that there exists $L\geq 1$ and $K'$ such that 
	$$\frac{1}{L}\dwp(X,Y)-K'\leq V_C(X,Y).$$
	The work of Schlenker \cite{Schlenker-MRL} (c.f. \cite{BC}), there is a constant $K''>0$ such that $V_C(X,Y)\leq \VR(X,Y) + K''$.
	Combined with the bound by Brock, we have
	$$\frac{1}{L}\dwp(X,Y)-K'\leq V_C(X,Y)\leq \VR(X,Y) + K''.$$
	By letting $K:=K'+K''$, we have the desired lower bound by (3) of Theorem \ref{thm.d_R-bound}.
	The upper bound is obtained in (1) of Theorem \ref{thm.d_R-bound}.
\end{proof}

Now we prove Theorem \ref{thm.TL=HV}.
The proof we demonstrate here utilizes some ergodic theory, which is inspired by Karlsson-Ledrappier \cite[Proof of Theorem 1.1]{KL}.
\begin{thm}
Let $\psi\in\mathrm{MCG}(S)$	 be a pseudo-Anosov mapping class and 
$M(\psi)$ the mapping torus of $\psi$.
Then the translation length $\tau_R(\psi)$ of $\psi$ with respect to $d_R$ is equal to the hyperbolic volume of the mapping torus $M(\psi)$, i.e. for any $X\in\T$, 
$$\tau_R(\psi):=\lim_{k\rightarrow\infty}\frac{d_R(X,\psi^k(X))}{k} = \Vol(M(\psi)).$$
\end{thm}
\begin{proof}
	By Theorem \ref{thm.lim-nu} and Theorem \ref{thm.d_R-bound},
	we have that $\tau_R(\psi)\geq \Vol(M(\psi))$.
	
	For the converse, first note that $\tau_R(\psi)$ is independent of $X\in\T$ by the triangle inequality.
	Hence we use our base point $b\in\T$, and let $(\nu_n,q_n)\in\Tvh$ be such that 
	$$d_R(b,\psi^n b) = \nu_n(b)-\nu_n(\psi^n b) = -\nu_n(\psi^n b).$$
	We define $F:\Tvh\rightarrow\R$ by $F(\nu,q) = -\nu(\psi^{-1}b)$.
	Recall that for any $(\nu,q)\in\Tvh$, we have 
	\begin{equation}\label{eq.cocycle-nu}
			-\nu(\psi^{-n}b) = -\sum_{i=0}^{n-1}\psi^i\nu(\psi^{-1}b) = \sum_{i=0}^{n-1}F(\psi^i(\nu,q)).		
	\end{equation}
	Now let us define a probability measure $\eta_n$ by 
	$$\eta_n=\frac{1}{n}\sum_{i=0}^{n-1}(\psi^i)_*\delta_{(\nu_n,q_n)},$$ where $\delta_{(\nu,q)}$ is the Dirac measure at $(\nu,q)$ on $\Tvh$.
	Then $\eta_n$ is a Borel provability measure on $\Tvh$, which satisfies
	$$\int F(\nu,q)d\eta_n(\nu,q) = \frac{1}{n}d_R(b,\psi^{-n}b)$$
	by (\ref{eq.cocycle-nu}).
	Notice that by the triangle inequality, we have
	$$\frac{1}{n}d_R(b,\psi^{-n}b)\geq \lim_{k\rightarrow\infty}\frac{1}{k}d_R(b,\psi^{-k}b)=\tau_R(\psi).$$
	Since $\Tvh$ is compact, by taking a subsequence if necessary, we may suppose $\eta_n$ converges weakly to some $\eta_\infty$.
	Then by the definition of the weak limit, we have
	$$\int F(\nu,q)d\eta_\infty(\nu,q)\geq \tau_R(\psi).$$
	Furthermore, the definition of $\eta_n$ implies that $\eta_\infty$ is $\psi$ invariant, i.e. $\psi_*\eta_\infty = \eta_\infty$.
	We now consider the space $\mathcal{M}_\psi(\Tvh)$ of $\psi$ invariant measures on $\Tvh$.
	As $\mathcal{M}_\psi(\Tvh)$ is convex and $\eta_\infty\in\mathcal{M}_\psi(\Tvh)$, the Krein-Milman theorem shows that there is an extreme point $\mu\in\mathcal{M}_\psi(\Tvh)$ such that
	$$\int F(\nu,q)d\mu(\nu,q)\geq \tau_R(\psi).$$
	The standard theory of ergodic measures says that ergodic measures in $\mathcal{M}_\psi(\Tvh)$ are precisely the extreme points.
	Therefore, we see that $\mu$ is ergodic.
	Then by the Birkhoff ergodic theorem, for $\mu$-a.e. $(\nu,q)\in\Tvh$ we have
	$$\lim_{n\rightarrow\infty}\frac{1}{n}\sum_{i=0}^{n-1}F(\psi^i(\nu,q)) = \int F(\nu,q)d\mu(\nu,q)\geq \tau_R(\psi).$$
	Hence, by (\ref{eq.cocycle-nu}), 
	$$\lim_{n\rightarrow\infty}\frac{1}{n}(-\nu(\psi^{-n}b))\geq \tau_R(\psi).$$
	On the other hand, as $-\nu(\psi^{-n}b)\leq d_R(b,\psi^{-n}b)$, we have 
	$$\lim_{n\rightarrow\infty}\frac{1}{n}(-\nu(\psi^{-n}b))\leq \tau_R(\psi).$$
	Hence we have the equality.
	As $\tau_R(\psi)\geq 0$, by Theorem \ref{thm.lim-nu} and Corollary \ref{coro.lim_nu}, we see that $\tau_R(\psi) = \Vol(M(\psi))$.
\end{proof}
\section{Questions}
The distance $d_R$ is still very mysterious.
Let us conclude the paper with some questions.

Similarly to the case of WP metric, as $(\T,d_R)$ is not complete, we may not use Hopf-Rinow Theorem to find geodesics.
\begin{question}
	Is $(\T,d_R)$ a geodesic space? 
\end{question}
As $d_R(\cdot,\cdot)\leq\Con\dwp(\cdot,\cdot)$, one easily sees that  $\hatT$ is contained in the completion of $(\T,d_R)$.
\begin{question}
	What is the metric completion of $(\T,d_R)$?
\end{question}
Given any $X,Y\in\T$, there is $(\nu,q)\in\Tvh$ such that
$d_R(X,Y) = \nu(X) - \nu(Y)$.
As we have discussed in \S \ref{sec.wp-grad}, such $(\nu,q)$ defines the WP gradient flow. 
\begin{question}
	Let $X,Y\in\T$. Suppose that $d_R(X,Y)=\nu(X) - \nu(Y)$ for some $(\nu,q)$.
	Then does the WP gradient flow $X_t$ starting at $X$ determined by $(\nu,q)$ pass through $Y$?
\end{question}
Also it is interesting to understand the action of pseudo-Anosov maps.
Let $\psi\in\mcg$ be pseudo-Anosov.
Then the axis of $\psi$ should be a geodesic of $d_R$ invariant under $\psi$.
\begin{question}
	Does every pseudo-Anosov map have a (unique?) geodesic axis?
\end{question}
The distance $d_R$ is quasi-isometric to $\dwp$ (Theorem \ref{thm.wp-dv}), and $(\T,\dwp)$ is CAT(0) \cite{Yamada}.
Although CAT(0)-ness is not invariant under quasi-isometry, we might expect:
\begin{question}
	Is $(\T,d_R)$ a CAT(0) space?
\end{question}
The horoboundaries may be used to identify isometry groups (see e.g. \cite{Wal}).
In \cite{Wal}, except for some sporadic cases, Walsh identified the isometry group of the Thurston metric with the so-called extended mapping class groups (see \cite{Wal} for the definition).

\begin{question}
	Is $\mathrm{Isom}(\T,d_R)$ equal to the extended mapping class group?
	What about self-maps on $\T$ preserving $\VR$?
\end{question}
Since we are taking supremum in the definition of $d_R$, several properties of $V_R$ (say, smoothness) is not a priori inherited to $d_R$.
Let us finish the paper with the following question.
\begin{question}
	Is there a Riemannian or a Finsler metric on $\T$ which defines $d_R$?
\end{question}

\end{document}